\documentclass[11pt, reqno]{amsart}
\usepackage[T1]{fontenc}
\usepackage{amsmath}
\usepackage{amssymb}
\usepackage{amsthm}
\usepackage[left=3.5cm, right=3.5cm, paperheight=11.9in]{geometry}
\usepackage{hyperref}
\usepackage{fancyhdr}
\usepackage{enumitem}
\usepackage{comment}
\usepackage{nicefrac}
\usepackage{bm}
\usepackage{mathrsfs}
\usepackage{graphicx}
\usepackage[utf8]{inputenc}
\usepackage{cancel}
\usepackage{mathtools}

\newtheorem{thm}{Theorem}[section]
\newtheorem{cor}[thm]{Corollary}
\newtheorem{lem}[thm]{Lemma}

\newtheorem{prop}[thm]{Proposition}

\theoremstyle{definition} 
\newtheorem{defi}[thm]{Definition}
\newtheorem{rmk}[thm]{Remark}
\newtheorem{example}[thm]{Example}
\numberwithin{equation}{section}

\theoremstyle{remark}
\newtheorem{claim}{\textsc{Claim}}
\newtheorem*{claim*}{\textsc{Claim}}

\newcounter{smallromans}

{\end{list}}

\AtBeginDocument{%
   \def\MR#1{}
}

\hypersetup{
    pdftitle={Strong universality and recurrence},
    pdfauthor={Paolo Leonetti},
    pdfmenubar=false,
    pdffitwindow=true,
    pdfstartview=FitH,
    colorlinks=true,
    linkcolor=blue,
    citecolor=green,
    urlcolor=cyan
}

\providecommand{\MR}[1]{}

\providecommand{\MR}{\relax\ifhmode\unskip\space\fi MR }

\providecommand{\href}[2]{#2}

\begin{document}
\title{Strong universality, recurrence, and analytic P-ideals in dynamical systems}

\author{Paolo Leonetti}
\address{Universit\'a degli Studi dell'Insubria, 
%Department of Economics, 
via Monte Generoso 71, 21100 Varese, Italy}
\address{
%Department of Statistics, 
Universit\'a Luigi Bocconi, via Roentgen 1, Milan 20136, Italy}
\email{leonetti.paolo@gmail.com}


%37B02 Dynamics in general topological spaces 
%37B20 Notions of recurrence and recurrent behavior in dynamical systems
%37B99 None of the above, but in this section (topological dynamics)
%47A16 Cyclic vectors, hypercyclic and chaotic operators
%%\date{\today}

\begin{abstract} 
\noindent 
Given a dynamical system $(X,T)$ and a family $\mathsf{I}\subseteq \mathcal{P}(\omega)$ of  \,\textquotedblleft small\textquotedblright\,  sets of nonnegative integers, a point $x \in X$ is said to be $\mathsf{I}$-strong universal if for each $y \in X$ there exists a subsequence $(T^nx: n \in A)$ of its orbit which is convergent to $y$ and, in addition, the set of indexes $A$ is \textquotedblleft not small,\textquotedblright\, 
that is, $A\notin \mathsf{I}$. 
An analoguous definition is given for $\mathsf{I}$-strong recurrence. In this work, we provide several 
structural properties and 
relationships between $\mathsf{I}$-strong universality, $\mathsf{I}$-strong recurrence, and the corresponding ordinary notions of $\mathsf{I}$-universality and $\mathsf{I}$-recurrence. 

As applications, 
%we show that if $X$ has infinite density character and $x \in X$ has dense $T$-orbit, 
%then there exists $y \in X$ with the property that every subsequence 
%$(T^{n}x: n\in A)$ is not convergent to $y$ whenever the set of indexes $A$ has positive upper Banach density. 
%In addition, 
we provide sufficient conditions which ensure the equivalence between the above notions and the property that each nonempty open set contains some cluster point of some orbit. %\textcolor{blue}{No result assumes the full linearity of $T$.} 
%In particular, we prove that if $x$ is not periodic, then the subsequence $(T^{n}x: n\in A)$ is not convergent to $x$ for all sets of indexes $A$ with positive upper Banach density. 
In addition, we show that if $T$ is a homomorphism on a Fr\'{e}chet space $X$ and there exists a dense set of vectors with null orbit, then for each $y \in X$ the set of all vectors $x \in X$ such that $\lim_{n \in A}T^nx=y$ for some $A\subseteq \omega$ with nonzero upper asymptotic density is either empty or comeager.  %None of our main results relies on the full linearity of $T$. 

In the special case of linear dynamical systems on Banach spaces with a dense set of uniformly recurrent vectors, we obtain that $T$ is upper frequently hypercyclic if and only if there exists a hypercyclic vector $x \in X$ for which $\lim_{n \in A}T^nx=0$ for some $A\subseteq \omega$ with nonzero upper asymptotic density. 
\end{abstract}
\maketitle
\thispagestyle{empty}

%-------------------------------------------------------%
%                                                       %
% 						INTRODUCTION 					%
%                                                       %
%-------------------------------------------------------%
\section{Introduction}\label{sec:intro}

In this work, a dynamical system denotes a pair $(X,T)$ where $X$ is a topological space and $T: X\to X$ is a continuous map. 
We are mainly concerned with the study of structural properties of nonlinear dynamical systems with respect to certain families of \textquotedblleft small\textquotedblright\, subsets of the nonnegative integers $\omega$, which we denote by $\mathsf{I}$. 
Regarding these families $\mathsf{I}$ as subsets of the Cantor space $\{0,1\}^\omega$, we can speak about their topological complexity: for instance, a family $\mathsf{I}$ can be an analytic, Borel, or a $F_\sigma$-subset of $\mathcal{P}(\omega)$. 
The first main novelty of this paper stands in the connection between the complexities of the families $\mathsf{I}$ and the structural properties of the underlying dynamical systems. 
The second one is the introduction and the study of stronger notions of universality and recurrence (precise definitions will be given in the next sections). %Lastly, we observe also that certain results which hold in the linear case can be extended in the general case. Of course, this is not done for the sake of generalization. 

As examples of applications of our results for \emph{linear} dynamical systems (that is, with $T$ continuous linear operator on a topological vector space $X$), we show that:
\begin{enumerate}[label={\rm (\roman{*})}]
\item If 
$X$ is nonzero and $x \in X$ is 
%a point $x \in X$ is 
%not periodic under $T$, 
a hypercyclic vector (i.e., it has a dense orbit under $T$), 
then the subsequence 
$(T^{n}x: n\in A)$ is \emph{not} convergent to $x$ for all sets of indexes $A$ with positive upper Banach density (see Theorem \ref{thm:nonexistencestronglyuniversal}). 
%$(T^{n_k}x: k \in \omega)$ is \emph{not} convergent to $y$ whenever the set of indexes $\{n_1,n_2,\ldots\}$ has positive upper Banach density (see Corollary \ref{cor:noZstronguniversal}). 
%\item If $X$ is a Fr\'{e}chet sequence space with a basis and $T$ is a chaotic weighted shift, then there exists a sequence $x\in X$ for which the return set $\{n \in \omega: T^nx \in U\}$ has bounded gaps for every nonempty set $U\subseteq X$ (see Corollary \ref{cor:strenghten}). 
\item If $X$ is a Fr\'{e}chet space and there exists a dense set of vectors with orbit convergent to $0$, then for each $y \in X$ the set of all vectors $x \in X$ such that $\lim_{n \in A}T^nx=y$ for some $A\subseteq \omega$ with nonzero upper asymptotic density is either empty or comeager (see Corollary \ref{cor:emptyordenselinear}). 
\item If $X$ is a Banach space and there exists a dense set of uniformly recurrent vector, then $T$ is upper frequently hypercyclic if and only if there exists a hypercyclic vector $x \in X$ such that $\lim_{n \in A}T^nx=0$ for some $A\subseteq \omega$ with positive upper asymptotic density (see Corollary \ref{cor:extensionhilbertcharactrization}).
\end{enumerate}
We refer the reader to \cite{MR2533318, MR2919812} for two excellent monographs on linear dynamical systems.

%All the results are consequences of the structural properties of return sets. 
%

In a moment, we define the notions of universality and recurrence with respect to an ideal $\mathsf{I}$ on $\omega$, together with their \textquotedblleft strong\textquotedblright\, variants (which are, informally, the exact analogues through subsequences in place of neighborhoods). Then, we provide several relationships between the latter ones, based on the structural properties of return sets. In doing so, we also ask to which extent several results which hold in the linear case can be carried to the nonlinear one. This is not done for the sake of generality, but to clarify what are the properties underlying them and, as a byproduct, to include possibly new examples. Not all our main theorems, in fact, rely on the full linearity of $T$. 

% \textcolor{red}{complete}

%The main results follow in Section \ref{sec:mainresults}. \textcolor{red}{[Complete comments, plan]}%This will allowus to discuss them formally and prove them.

\subsection{Preliminaries on ideals} 
%\textcolor{red}{Aggiungere commento con le Furstenberg families!} 
Let $\mathsf{I}\subseteq \mathcal{P}(\omega)$ be an ideal, that is, a family a subsets of the nonnegative integers $\omega$ closed under subsets and finite unions. 
%Unless otherwise stated, 
It is also assumed that $\mathsf{I}$ contains the family $\mathrm{Fin}$ of finite sets and it is different from the power set $\mathcal{P}(\omega)$. 
As usual, $\mathcal{P}(\omega)$ is identified with the Cantor space $\{0,1\}^\omega$ (i.e., the topological product of countably many $\{0,1\}$, each one with the discrete topology), hence a basic clopen set of $\mathcal{P}(\omega)$ will be a cylinder of the type $\{S\subseteq \omega: S\cap [0,\max F]=F\}$ for some nonempty $F \in \mathrm{Fin}$. For instance, since the topology on $\mathcal{P}(\omega)$ is Polish (hence, Hausdorff) and $\mathrm{Fin}$ contains countably many elements, then $\mathrm{Fin}$ is a $F_\sigma$ ideal. 
For notational convenience, let
$$
\mathsf{I}^\star:=\{\omega \setminus S: S \in \mathsf{I}\} 
\quad \text{ and }\quad 
\mathsf{I}^+:=\mathcal{P}(\omega)\setminus \mathsf{I}
$$
be the dual filter of $\mathsf{I}$ and the the family of $\mathsf{I}$-positive sets, respectively. 
%$\mathsf{I}^\star:=\{S\subseteq \omega: S^c \in \mathsf{I}\}$ be its dual filter and $\mathsf{I}^+:=\mathcal{P}(\omega)\setminus \mathsf{I}$ be the family of $\mathsf{I}$-positive sets. 
%Given ideals $\mathcal{I}, \mathcal{J}$ on $\mathbf{N}$, we let their Fubini product be the ideal on $\mathbf{N}^2$ defined by 
%$$
%\mathcal{I}\times \mathcal{J}:=\{A\subseteq \mathbf{N}^2: \{n \in \mathbf{N}: \{k \in \mathbf{N}: (n,k) \in A\}\notin \mathcal{J}\} \in \mathcal{I}\}.
%$$ 
We refer the reader to \cite{MR1711328, Mezathesis} for two excellent monograph on the theory of ideals on $\omega$. 

Notable examples of ideals include $\mathrm{Fin}$ and the families of sets with logathmic density zero, asymptotic density zero, and Banach density zero, i.e., $\mathsf{L}:=\left\{S\subseteq \omega: \mathsf{ld}^\star(S)=0\right\}$, and similarly
$$
%\mathsf{L}:=\left\{S\subseteq \omega: \mathsf{ld}^\star(S)=0\right\},
\mathsf{Z}:=\left\{S\subseteq \omega: \mathsf{d}^\star(S)=0\right\} 
\quad \text{ and }\quad 
\mathsf{B}:=\left\{S\subseteq \omega: \mathsf{bd}^\star(S)=0\right\}; 
$$
here, $\mathsf{ld}^\star(S):=\limsup_n \sum_{k \in S\cap (0,n]}1/k\log(n)$  stands for the upper logarithmic density of $S\subseteq \omega$, and similarly $\mathsf{d}^\star(S):=\limsup_{n}\frac{1}{n}|S\cap [0,n)|$ for the upper logarithmic density and $\mathsf{bd}^\star(S):=\lim_{n} \max_{k} \frac{1}{n}|S\cap [k,k+n)|$ for the upper Banach density. It is 
immediate that 
%well known 
%(and easy to show) 
$\mathsf{B}\subseteq \mathsf{Z}\subseteq \mathsf{L}$ since $\mathsf{ld}^\star\le \mathsf{d}^\star \le \mathsf{bd}^\star$, cf. \cite{MR4054777}. 
Also, it is not difficult to see that $\mathsf{L}$, $\mathsf{Z}$, and $\mathsf{B}$ are $F_{\sigma\delta}$ ideals. % which are not $F_\sigma$. 
%Remarkably, both $\mathsf{L}$ and $\mathsf{Z}$ are $P$-ideals (as defined in Section \ref{sec:furstenberg} below), while $\mathsf{B}$ is not. 
Many additional examples of \textquotedblleft well-behaved\textquotedblright\, ideals will be given in Section \ref{sec:furstenberg}. 
Lastly, an ideal $\mathsf{I}$ is maximal (with respect to inclusion) if and only if 
$\mathsf{I}^\star=\mathsf{I}^+$, that is, if and only if $\mathsf{I}^\star$ is a free ultrafilter. % on $\omega$. 
%Lastly, it is simple to show, by an application of Baire's category theorem, that an ideal $\mathsf{I}$ cannot be $G_\delta$. 
%the summable ideal
%$$
%\mathsf{I}_{1/n}:=\left\{S\subseteq \omega: \sum\nolimits_{n \in S}\frac{1}{n+1}<\infty\right\},
%%$$
%and the family of uniform density zero sets 
%$$
%\mathsf{B}:=\left\{S\subseteq \omega: \mathsf{bd}^\star(S)=0\right\},
%$$
%where $\mathsf{bd}^\star(S):=\lim_{n} \max_{k} \frac{1}{n}|S\cap [k,k+n)|$ denotes the upper Banach density of $S\subseteq \omega$. 

As anticipated, the ideal $\mathsf{I}$ represents the family of \textquotedblleft small\textquotedblright\, sets of integers. In the recent literature, stronger notions of recurrence and universality (and, in particular, hypercyclicity) have been considered with respect to 
upward directed families $\mathsf{F}\subseteq \mathcal{P}(\omega)$, which are usually called Furstenberg families, see Section \ref{sec:furstenberg} below for details. 
%e.g. \cite{MR3847081}. 
%Furstenberg families $\mathsf{F}$  which are, simply, upward directed subsets of $\mathcal{P}(\omega)$. 
Since $\mathsf{F}$ represents the family of sets of integers which are \textquotedblleft not small,\textquotedblright\, in our case it is enough to define equivalently $\mathsf{F}=\mathsf{I}^+$ with the caveat that we require the additional property that $A\cup B\notin \mathsf{F}$ for all $A,B \notin \mathsf{F}$, so that \textquotedblleft the union of small sets is small\textquotedblright\, (for instance, the family $\mathsf{Z}^+$ of sets with positive upper asymptotic density satisfies it, while the family of sets with positive lower asymptotic density does not; another Furstenberg family which does not satisfy it will be given in Example \ref{example:syndetic} below). 
As we are going to see, this natural additional condition often allows us to obtain (positive) analogues of results which hold in the ordinary case $\mathsf{I}=\mathrm{Fin}$. 
Lastly, we remark that this additional property on $\mathsf{F}$ already played a significant role in linear dynamics: for instance, it has been called \textquotedblleft partition regular\textquotedblright\, in \cite{MR4489276}, or \textquotedblleft Ramsey property\textquotedblright\, in \cite{MR4398486}. 

%\textcolor{red}{Reference risultato dopo}

%Ideals are regarded as subsets of the Cantor space $\{0,1\}^{\omega}$, hence it is possible to speak about their topological complexities. In particular, an ideal can be $F_\sigma$, $F_{\sigma\delta}$, Borel, analytic, meager, etc.
%We refer to 
%\cite{MR2777744, MR3920747} 
%for recent surveys on ideals and associated filters.
%\cite{MR2777744} 
%for a recent survey on ideals and associated filters. 

%Given an ideal $\mathsf{I}$ on $\omega$ and a sequence $\bm{x}=(x_n)$ taking values in a topological space $X$, we say that $\eta \in X$ is an $\mathsf{I}$\emph{-cluster points} of $\bm{x}$ if 
%$$
%\{n \in \omega: x_n \in U\} \in \mathsf{I}^+
%$$
%for all open sets $U\subseteq X$ containing $\eta$. 
%
%Let $X$ be a topological space and, for each $\eta \in X$, denote by $\mathcal{O}(\eta)$ the family of open sets $U\subseteq X$ containing $\eta$. 
%Given and ideal $\mathsf{I}$ on $\omega$ and a sequence $\bm{x}=(x_n)$ taking values in $X$, we denote by $\Gamma_{\bm{x}}(\mathsf{I})$ the set of its $\mathsf{I}$\emph{-cluster points}, that is, 
%$$
%\Gamma_{\bm{x}}(\mathsf{I}):=\left\{\eta \in X: \{n \in \omega: x_n \in U\} \notin \mathsf{I} \text{ for all }U \in \mathcal{O}(\eta)\right\}.
%$$
%the set of all $\eta \in X$ such that 
%$$
%\forall U \in \mathcal{O}(\eta), \quad 
%\{n \in \omega: x_n \in U\} \in \mathsf{I}^+.
%$$
%%%for all open neighborhoods $U$ of $\eta$. 

\subsection{Cluster points, universality, and recurrence} 
Given an ideal $\mathsf{I}$ on $\omega$ and a sequence $\bm{x}=(x_n)$ taking values in a topological space $X$, we say that $\eta \in X$ is an $\mathsf{I}$\emph{-cluster point} of $\bm{x}$ if 
$$
\{n \in \omega: x_n \in U\} \in \mathsf{I}^+
$$
for all open sets $U\subseteq X$ containing $\eta$. The set of $\mathsf{I}$-cluster points is denoted by $\Gamma_{\bm{x}}(\mathsf{I})$. In the literature, $\mathsf{Z}$-cluster points are usually \textquotedblleft statistical cluster points,\textquotedblright\, see e.g. \cite{MR1181163}. Several properties and characterizations of $\mathsf{I}$-cluster points can be found in \cite{MR3920799}. 

Let $\mathcal{T}=(T_n: n \in \omega)$ be a sequence of continuous maps $T_n: X\to Y$, where $X$ and $Y$ are two topological spaces, and, for each $x \in X$ and for each nonempty open set $U\subseteq X$, 
%let us write 
define the sequence $\mathcal{T}x$ and the return set $N_{\mathcal{T}}(x,U)$ by
$$
\mathcal{T}x:=\left(T_nx: n \in \omega\right)
\quad \text{ and }\quad 
N_{\mathcal{T}}(x,U):=\{n \in \omega: T_nx \in U\},
$$
respectively. If $X=Y$ and $\mathcal{T}=(T^n: n \in \omega)$ for some continuous map $T: X\to X$, we write simply $N_T(x,U)$ or $N(x,U)$ if there is no risk of confusion. 
% (and similarly for the other notions defined below). 
Note that, in the latter case, $\mathcal{T}x$ is the ordinary orbit %of $x$ with respect to $T$, i.e., 
$
\mathrm{orb}(x,T):=(x,Tx,T^2x,\ldots).
$
%of the dynamical system $(X,T)$. 

\begin{defi}\label{defi:Ihypercyclic}
Let $\mathsf{I}$ be ideal on $\omega$ and $\mathcal{T}=(T_n: n \in \omega)$ be a sequence of continuous maps $T_n: X\to Y$, where $X$ and $Y$ are topological spaces. 
We say that $x \in X$ is an $\bm{\mathsf{I}}$\textbf{-universal point} if every $y \in Y$ is an $\mathsf{I}$-cluster point of the sequence $\mathcal{T}x$, that is, 
$$
\Gamma_{\mathcal{T} x}(\mathsf{I})=Y.
$$
We denote the set of $\mathsf{I}$-universal points by $\mathrm{Univ}_{\mathcal{T}}(\mathsf{I})$. The sequence $\mathcal{T}$ is said to be $\mathsf{I}$-universal if it admits an $\mathsf{I}$-universal point. 
\end{defi}

Hereafter, $\mathrm{Fin}$-universal points are simply called universal points. In addition, a dynamical system $(X,T)$ is said to be $\mathsf{I}$-universal if $(T^n: n \in \omega)$ is $\mathsf{I}$-universal and we write $\mathrm{Univ}_{T}(\mathsf{I}):=\mathrm{Univ}_{(T^n)}(\mathsf{I})$. A similar remark applies for the other definitions below.

Note that a point $x \in X$ is $\mathsf{I}$-universal if and only if $N_{\mathcal{T}}(x,V) \in \mathsf{I}^+$ for all nonempty open sets $V\subseteq Y$. In the special case of linear dynamical systems $(X,T)$, the term \textquotedblleft universal\textquotedblright\, is commonly replaced by \textquotedblleft hypercyclic\textquotedblright\, (we are going to use this convention also in this work, see e.g. Corollary \ref{cor:extensionhilbertcharactrization} below). 
Accordingly, the vectors in $\mathrm{Univ}_{T}(\mathsf{I})$ are usually called: hypercyclic if $\mathsf{I}=\mathrm{Fin}$, reiteratively hypercyclic if $\mathsf{I}=\mathsf{B}$, and upper frequently hypercyclic if $\mathsf{I}=\mathsf{Z}$, see e.g. \cite{MR3552249} and references therein. 

Now, suppose that $\mathcal{T}$ is $\mathsf{I}$-universal. Then $Y$ has to be separable: indeed, the image of $\mathcal{T}x$ is dense in $Y$ for every $\mathsf{I}$-universal point $x$. 
Another simple observation is that, if $Y$ has two disjoint nonempty open sets $V_1,V_2\subseteq Y$ then $\mathsf{I}$ is not maximal: for, $N_{\mathcal{T}}(x,V_1)$ and $N_{\mathcal{T}}(x,V_2)$ are two disjoint $\mathsf{I}$-positive sets for every $\mathsf{I}$-universal point $x$. 
A stronger necessary condition on the structure of ideal $\mathsf{I}$ (that is, the existence of a sequence of \textquotedblleft sufficiently sparse\textquotedblright\, sequence of $\mathsf{I}$-positive sets) will be given below in Proposition \ref{prop:lipmaps}. At this point, one may wonder if $\mathsf{I}$ has necessarily a low topological complexity. For instance, is it true that if $\mathcal{T}$ is $\mathsf{I}$-universal then $\mathsf{I}$ is analytic? We answer it negatively in Proposition \ref{prop:nonborel}. 

In the case $X=Y$, we define the weaker notion of recurrence: 
\begin{defi}\label{defi:Irecurrent}
Let $\mathsf{I}$ be ideal on $\omega$ and $\mathcal{T}=(T_n: n \in \omega)$ be a sequence of continuous selfmaps $T_n: X\to X$, where $X$ is a topological space. 
We say that $x \in X$ is an $\bm{\mathsf{I}}$\textbf{-recurrent point} if $x$ is an $\mathsf{I}$-cluster point of the sequence $\mathcal{T}x$, that is, 
$$
x \in \Gamma_{\mathcal{T} x}(\mathsf{I}).
$$
We denote the set of $\mathsf{I}$-recurrent points by $\mathrm{Rec}_{\mathcal{T}}(\mathsf{I})$. The sequence $\mathcal{T}$ is said to be $\mathsf{I}$-recurrent if 
$\mathrm{Rec}_{\mathcal{T}}(\mathsf{I})$ is dense in $X$.
%it admits a $\mathsf{I}$-recurrent point. 
\end{defi}

In other words, a point $x \in X$ is $\mathsf{I}$-recurrent whenever $N_{\mathcal{T}}(x,U) \in \mathsf{I}^+$ for all open sets $U\subseteq X$ containing $x$. 
Note that trivial examples such that the identity map on $X$ show that $\mathsf{I}$-recurrence does not coincide with $\mathsf{I}$-universality. 
We refer to \cite{MR4489276, CardecciaMuro, MR4529811, MR3275437, MR4523465, Grivaux2023} for recent results on $\mathsf{I}$-recurrent points in the context of linear dynamical systems. % $(X,T)$ where $X$ is a Fr\'{e}chet space. 

%\textcolor{red}{COMPLETE}

\subsection{Limit points, strong universality and strong recurrence} 

Given an ideal $\mathsf{I}$ on $\omega$ and a sequence $\bm{x}=(x_n)$ taking values in a topological space $X$, we say that $\eta \in X$ is an $\mathsf{I}$\emph{-limit point} of $\bm{x}$ if there exists a subsequence $(x_{n_k})$ such that 
$$
\lim\nolimits_kx_{n_k}=\eta 
\quad \text{ and }\quad 
\{n_k: k \in \omega\} \in \mathsf{I}^+.
$$
The set of $\mathsf{I}$-limit points is denoted by $\Lambda_{\bm{x}}(\mathsf{I})$. 

Since ideals $\mathsf{I}$ are invariant modulo finite sets, it is clear that every $\mathsf{I}$-limit point is always an $\mathsf{I}$-cluster point, so that $\Lambda_{\bm{x}}(\mathsf{I})\subseteq \Gamma_{\bm{x}}(\mathsf{I})$. However, even for well-behaved topological spaces and ideals, they can be extremely different. As an example, if $\alpha$ is an irrational number and $\bm{x}$ is the real sequence such that $x_n$ is the fractional part of $n\alpha$ for each $n \in \omega$, then $\Gamma_{\bm{x}}(\mathsf{Z})=[0,1]$ while $\Lambda_{\bm{x}}(\mathsf{Z})=\emptyset$, cf. \cite[Example 3]{MR1181163}. In the same direction, it is known that the set of $\mathsf{Z}$-limit points is always a $F_\sigma$-set containing all isolated $\mathsf{Z}$-cluster points; and, conversely, given $A\subseteq B\subseteq \mathbb{R}$ where $A$ is a nonempty $F_\sigma$-set containing the isolated points of a closed set $B$, there exists a real sequence $\bm{x}$ such that 
$\Lambda_{\bm{x}}(\mathsf{Z})=A$ and $\Gamma_{\bm{x}}(\mathsf{Z})=B$, see \cite[Corollary 3.3]{MR3883171}.

%The relationship between $\mathsf{I}$-cluster points and $\mathsf{I}$-limit points has been studied in \cite{MR3883171}. 
%cf. also %LAVORI XI 

This premise motivates the following stronger definitions: 
\begin{defi}\label{defi:strongIuniversal}
With the same hypotheses of Definition \ref{defi:Ihypercyclic}, we say that $x \in X$ is a $\bm{\mathsf{I}}$\textbf{-strong universal point} if every $y \in Y$ is an $\mathsf{I}$-limit point of the sequence $\mathcal{T}x$, that is, 
$$
\Lambda_{\mathcal{T}x}(\mathsf{I})=Y.
$$
We denote the set of $\mathsf{I}$-strong universal points by $\mathrm{sUniv}_{\mathcal{T}}(\mathsf{I})$. The sequence $\mathcal{T}$ is said to be $\mathsf{I}$-strong universal if it admits an $\mathsf{I}$-strong universal point. 
\end{defi}
Equivalently, a point $x \in X$ is $\mathsf{I}$-strong universal provided that for each $y \in Y$ there exists a subsequence $(T_{n_k}x: k \in \omega)$ which is convergent to $y$ and $\{n_k: k \in \omega\} \in \mathsf{I}^+$. 
Note that the notion of\, \textquotedblleft sufficiently large subsequence\textquotedblright\, is not new in (linear) topological dynamics: in \cite{MR3084499}, for instance, the authors study the existence of vectors $x$ in separable Banach spaces such that $\lim_{n \in A}T^nx=0$ and $\lim_{n \in B}\|T^nx\|=\infty$, for some $A,B\subseteq \omega$ with $\mathsf{d}^\star(A)=\mathsf{d}^\star(B)=1$; see also \cite{MR3334899, MR4508017}, \cite[Section 4]{MR3255465}, and \cite[Theorem 5.23]{MR4238631}.

\begin{defi}\label{defi:Istrongrecurrent}
With the same hypotheses of Definition \ref{defi:Irecurrent}, we say that $x \in X$ is a $\bm{\mathsf{I}}$\textbf{-strong recurrent point} if $x \in X$ is an $\mathsf{I}$-limit point of the sequence $\mathcal{T}x$, that is, 
$$
x \in \Lambda_{\mathcal{T}x}(\mathsf{I}).
$$
We denote the set of $\mathsf{I}$-strong universal points by $\mathrm{sRec}_{\mathcal{T}}(\mathsf{I})$. The sequence $\mathcal{T}$ is said to be $\mathsf{I}$-strong recurrent if 
$\mathrm{sRec}_{\mathcal{T}}(\mathsf{I})$ is dense in $X$.
%it admits an $\mathsf{I}$-strong recurrent point. 
\end{defi}

It is immediate that $\mathsf{I}$-strong universality implies both $\mathsf{I}$-universality and $\mathsf{I}$-strong recurrence, and that each of the latter ones implies $\mathsf{I}$-recurrence. In addition, $\mathsf{I}$-universality implies $\mathsf{J}$-universality whenever $\mathsf{J}\subseteq \mathsf{I}$, and analogously for the other notions. 

Of course, if $X$ is first countable, $\mathrm{Fin}$-strong universality and $\mathrm{Fin}$-strong recurrence coincide with $\mathrm{Fin}$-universality and $\mathrm{Fin}$-recurrence, respectively, cf. \cite[Lemma 3.1(ii)]{MR3920799}. 
However, as we are going to show later there is a major drawback:  the notions of $\mathsf{I}$-strong recurrence and $\mathsf{I}$-strong universality are satisfied by only somewhat trivial examples whenever $\mathsf{I}$ is sufficiently large (more precisely, if $\mathsf{B}\subseteq \mathsf{I}$, hence including the case of the ideal $\mathsf{Z}$ of asymptotic density zero sets), see Corollary \ref{thm:nonexistencebdstronguniversal} and Corollary \ref{cor:noZstronguniversal} below. 

%\textcolor{red}{COMMENTS, HEURISTIC E COINCIDONO PER F SIGMA}

\subsection{Plan} In Section \ref{sec:furstenberg}, we show that certain (complements of) ideals on $\omega$ with low topological complexity are upper Furstenberg families. We proceed in Section \ref{sec:structuralbasics} with the discussion of the topological structure of the sets defined above, and necessary conditions on $\mathsf{I}$ for the existence of $\mathsf{I}$-universal dynamical systems. In Section \ref{sec:returnsets}, starting from structural properties of return sets, we show characterizations of $\mathsf{B}$-universality (Theorem \ref{thm:reiterativeequalhypercyclic}), characterizations of $\mathsf{I}$-universality for left-translation invariant ideals $\mathsf{I}$ (Theorem \ref{thm:upperfrequentcharacterization}), and the equivalence between $\mathsf{B}$-strong recurrence and periodicity (Theorem \ref{thm:nonexistencestronglyuniversal}). 

We continue in Section \ref{sec:applications} with further structural results and several applications. 
We start by showing that, if $T$ is additive and there exists a large supply of points with null orbits, then for all $x \in X$ and all nonempty open sets $U,V\subseteq X$ there exists $y \in V$ such that the return set $N(y,U)$ contains $N(x,U_0)$ modulo finite sets, for some $U_0\subseteq U$ (Proposition \ref{prop:endomorphism}). 
%We apply this result to prove that certain dynamical systems have a comeager set of uniformly recurrent vectors (Theorem \ref{thm:strenghten}). In addition, 
Using the topological Banach--Mazur game, we show that the set of elements with a prescribed $\mathsf{I}$-limit point is either empty or comeager (Theorem \ref{thm:denseIimitpoints}). 
This allows also to prove, for certain ideals $\mathsf{I}$, the equivalence between the notions of $\mathsf{I}$-universality, $\mathsf{I}$-recurrence, their strong variants, and the property that each nonempty open set contains an $\mathsf{I}$-cluster point of some orbit (Theorem \ref{thm:zerooneorbitlkfj}). The latter condition is reminescent of the so-called \textquotedblleft zero-one law of orbital limit points,\textquotedblright\, see \cite{BonillaErdmann, MR2881542}. 
Additional equivalences for $\mathsf{I}$-universality are given in Corollary \ref{cor:extensionhilbertcharactrization} through the study of a parameter introduced in \cite{MR3255465, MR4238631} for some special instances. To conclude, we provide an Ansari-type result for arithmetic ideals (Theorem \ref{thm:generalizedansari}). 
%[...] \textcolor{red}{[Completare]}

%\textcolor{red}{[COMPLETE, perchè prima sez2, controllare la domanda 3 di Shkarin che ogni summable ideal ammette un operatore di quel tipo iperciclico]}

\subsection{Notations} Given sets $A,B\subseteq \omega$ and an integer $k \in\omega$, we write 
%$k\cdot A:=\{ka: a \in A\}$, 
$A-k:=\{n \in \omega: n=a-k \text{ for some }a \in A\}$ and 
$A-B:=\bigcup_{k \in B}(A-k)$.  The sets $A+k$ and $A+B$ are defined analogously. In addition, $k\cdot A:=\{ka: a \in A\}$. The closure of a set $S$ in a given topological space is denoted by $\overline{S}$.

%We denote by $\mathrm{dens}(X)$ and $\mathrm{w}(X)$ the density character and the weight of a topological space $X$, respectively, that is, the smallest cardinality of a dense subset and of a basis of $X$, respectively. 
%We denote by $\mathrm{dens}(X)$ the density character of a topological space $X$, that is, the smallest cardinality of a dense subset of $X$. 
%In addition, if $T: X\to X$ is continuous, 
%%If $T$ is a continuous selfmap on a topological space $X$, 
%we write $\mathrm{Fix}(T)$ for the set of fixed points of $T$ and $\mathrm{Per}(T)$ for the set of period points. \textcolor{red}{PLAN, CLOSURE}

\section{Furstenberg families, $F_\sigma$ ideals, and Analytic P-ideals}\label{sec:furstenberg}

Let us recall the definition of upper Furstenberg families introduced in \cite{MR3847081}:
\begin{defi}\label{defi:upperFurstenberg}
A Furstenberg family (that is, an upward directed family) $\mathsf{F}\subseteq \mathcal{P}(\omega)$ is said to \emph{upper} if it is admits a representation 
\begin{equation}\label{eq:upperFurstenberg}
\mathsf{F}=\bigcup\nolimits_{i \in I}\bigcap\nolimits_{j \in \omega} \mathsf{F}_{i,j},
\end{equation}
for some arbitrary index set $I$ and families $\mathsf{F}_{i,j} \subseteq \mathcal{P}(\omega)$, with the following properties:
\begin{enumerate}
[label={\rm (\textsc{f}\arabic{*})}]
%[label={\rm (\roman{*})}]
\item \label{item:1upper} each $\mathsf{F}_{i,j}$ is \emph{finitely hereditary upward}, that is, for all $A \in \mathsf{F}_{i,j}$ there exists $F \in \mathrm{Fin}$ for which $B \in \mathsf{F}_{i,j}$ whenever $A\cap F\subseteq B$;
\item \label{item:2upper} $\mathsf{F}$ is \emph{uniformly left-translation invariant}, that is, for all $A \in \mathsf{F}$ there exists $i \in I$ such that $A-k \in \mathsf{F}_{i,j}$ for all $j,k \in \omega$.
\end{enumerate}
In addition, an upward directed family $\mathsf{F}$ with representation \eqref{eq:upperFurstenberg} is said to be \emph{uniformly finitely invariant} if:
\begin{enumerate}
[label={\rm (\textsc{f}\arabic{*})}]
%[label={\rm (\roman{*})}]
\setcounter{enumi}{2}
\item \label{item:3upper} for all $A \in \mathsf{F}$ there exists $i \in I$ such that $A\setminus [0,k] \in \mathsf{F}_{i,j}$ for all $j,k \in \omega$.
\end{enumerate}
\end{defi}

Even if Definition \ref{defi:upperFurstenberg} looks quite convoluted, we are going to show that it is satisfied by a really large family of examples. 
Several ones were already given by Bonilla and Grosse-Erdmann in \cite[Example 2.11]{MR3847081}, including the complements of the ideals $\mathrm{Fin}$, $\mathsf{B}$, $\mathsf{Z}$, the summable ideal
\begin{equation}\label{eq:defsummable}
\mathsf{I}_{(a_n)}:=\left\{S\subseteq \omega: \sum_{n \in S}a_n<\infty\right\}
\end{equation}
where $(a_n)$ is a given nonnegative real sequence with divergent series, cf. \cite[Section 5]{MR2439433}, and the ideal 
\begin{equation}\label{eq:defnonngetaiveregular}
\mathsf{I}_A:=\left\{S\subseteq \omega: \lim_{n\to \infty}\sum_{k \in S}a_{n,k}=0\right\}
\end{equation}
where $A=(a_{n,k}: n,k \in\omega)$ is an infinite regular matrix with nonnegative real entries. More precisely, we are going to show in Theorem \ref{thm:Fsigmaupper} and Theorem \ref{thm:analyticPidealupper}  that the complements of ideals $\mathsf{I}_{(a_n)}$ and $\mathsf{I}_A$ are Furstenberg families which admits a representation as in \eqref{eq:upperFurstenberg} satisfying \ref{item:1upper} and \ref{item:3upper}. In addition, we identify conditions under which they satisfy also \ref{item:2upper} (for instance, $\mathsf{I}_{(a_n)}^+$ satisfies \ref{item:2upper} if $(a_n)$ is weakly decreasing). 

We provide below further examples of upper Furstenberg families based on the topological complexity of $\mathsf{I}$. 
To this aim, recall that $\mathsf{I}$ is a \emph{P-ideal} if it is $\sigma$-directed modulo finite sets, that is, for every sequence $(S_n)$ of sets in $\mathsf{I}$ there exists $S \in \mathsf{I}$ such that $S_n\setminus S$ is finite for all integers $n\in \omega$. 
Moreover, an ideal $\mathsf{I}$ is said to be \emph{right-translation invariant} if $S+k \in \mathsf{I}$ for all $S \in \mathsf{I}$ and $k \in \omega$. Equivalently, since ideals are invariant modulo finite sets, $S-k \in \mathsf{I}^+$ for all $S \in \mathsf{I}^+$ and $k \in \omega$. The definition of \emph{left-translation invariant} ideal, which will be needed later, is analogous.

Before we proceed, note that an ideal on $\omega$ is neither open nor closed and, in addition, it cannot be $G_\delta$ by a simple application of Baire's category theorem, see e.g. \cite[Proposition 1.2.1]{Mezathesis}. It follows that the smallest possible complexity of $\mathsf{I}$ is $F_\sigma$ (which is the case of $\mathrm{Fin}$). Lastly, it is well known (and not difficult to show) that the ideal $\mathsf{B}$ is a $F_{\sigma\delta}$-ideal which is not $F_\sigma$. In addition, $\mathsf{B}$ is not a $P$-ideal, see e.g. \cite[p. 299]{MR632187}.

\subsection{$F_\sigma$-ideals} 
It is immediate to see that all summable ideals $\mathsf{I}_{(a_n)}$ defined in \eqref{eq:defsummable} are $F_\sigma$-ideals which, in addition, are also $P$-ideals (note that this includes also the case $\mathsf{I}=\mathrm{Fin}$ by setting $a_n=1$ for all $n\in \omega$). At the same time, the class of summable ideals is a rather special class of $F_\sigma$-ideals. Examples of $F_\sigma$ $P$-ideals which are not summable are given in \cite[Section 1.11]{MR1711328}. It is easy to see that there are $F_\sigma$-ideals without the $P$-property: the Fubini product $\mathrm{Fin}\times \emptyset$, which can be defined as 
$$\{S\subseteq \omega: \{n \in \omega: \{S\cap I_n\} \neq \emptyset\}\in \mathrm{Fin}\},$$ 
where $(I_n)$ is a given partition of $\omega$ into infinite sets, is a $F_\sigma$-ideal which is not a $P$-ideal, see \cite[Example 1.2.3(a)]{MR1711328}. A nonexhaustive list of additional examples  include the Fubini sum $\mathrm{Fin}\oplus \mathcal{P}(\omega)=\{S\subseteq \omega: S \cap (2\cdot \omega) \in \mathrm{Fin}\}$, Solecky's ideal \cite[p. 10]{Mezathesis}, the eventually different ideal \cite[p. 12]{Mezathesis}, 
%fragmented ideals , e' di HRUSAK (vedi LImit points of subsequences}
Tsirelson's ideals \cite{MR1680630},
and pathological examples as in \cite[pp. 31--33]{MR1711328}. 

Hereafter, a \emph{lower semicontinuous submeasure} (in short, \emph{lscsm}) is a function $\varphi: \mathcal{P}(\omega) \to [0,\infty]$ such that: 
\begin{enumerate}[label={\rm (\roman{*})}]
\item $\varphi(\emptyset)=0$, 
\item $\varphi(A)\le \varphi(B)$ for all $A\subseteq B\subseteq \omega$,
\item $\varphi(A\cup B)\le \varphi(A)+\varphi(B)$ for all $A,B\subseteq \omega$, 
\item $\varphi(F)<\infty$ for all $F \in \mathrm{Fin}$, and 
\item $\varphi(A)=\sup\{\varphi(A\cap F): F \in \mathrm{Fin}\}$ for all $A\subseteq \omega$. 
\end{enumerate}
By a classical result due to Mazur \cite{MR1124539}, an ideal $\mathsf{I}$ on $\omega$ is a $F_\sigma$-ideal if and only if there exists a lscsm $\varphi$ such that 
\begin{equation}\label{eq:mazurcharacter}
\mathsf{I}=\mathrm{Fin}(\varphi):=\{S\subseteq \omega: \varphi(S)<\infty\} 
\,\,\,\,\,\text{ and }\,\,\,\,\,
\varphi(\omega)=\infty. 
\end{equation}
For instance, if $\mathsf{I}_{(a_n)}$ is a summable ideal as in \eqref{eq:defsummable}, then $\mathsf{I}_{(a_n)}=\mathrm{Fin}(\varphi)$, where $\varphi(S)=\sum_{n \in S}a_n$ for all $S\subseteq \omega$. 

\begin{thm}\label{thm:Fsigmaupper}
Let $\mathsf{I}$ be a $F_\sigma$-ideal on $\omega$. Then $\mathsf{I}^+$ is a Furstenberg family which admits a representation as in \eqref{eq:mazurcharacter} satisfying \ref{item:1upper} and \ref{item:3upper}. 

If, in addition, $\mathsf{I}$ is right-translation invariant then $\mathsf{I}^+$ satisfies also \ref{item:2upper}.
\end{thm}
\begin{proof}
Let $\varphi$ be a lscsm such that $\mathsf{I}=\mathrm{Fin}(\varphi)$ as in \eqref{eq:mazurcharacter}. Set $I=\{0\}$ and define 
$$
\mathsf{F}_{0,j}:=\left\{S\subseteq	 \omega:\varphi(S)>2^j\right\}
$$
for all $j \in \omega$. Then it is immediate to see that 
$$
\mathsf{I}^+=\{S\subseteq \omega: \varphi(S)=\infty\}=\bigcup\nolimits_{i \in I}\bigcap\nolimits_{j \in \omega} \mathsf{F}_{i,j}.
$$
At this point, fix $j \in \omega$ and pick $A \in \mathsf{F}_{0,j}$. Since $\varphi$ is a lscsm, there exists an integer $n \in \omega$ such that $\varphi(A\cap [0,n])>2^j$. By the monotonicity of $\varphi$, it follows that $B \in \mathsf{F}_{0,j}$ whenever $A\cap [0,n]\subseteq B$, which proves \ref{item:1upper}. The property \ref{item:3upper} is vacuously true since $I$ is a singleton and the ideal $\mathsf{I}$ is invariant under finite modifications. 

Similarly, if the ideal $\mathsf{I}$ is, in addition, right-translation invariant, then $S-k \in \mathsf{I}^+$ whenever $S \in \mathsf{I}^+$, which implies \ref{item:2upper}.
\end{proof}

%%%%%%%%%%%%%%%%
\subsection{Analytic $P$-ideals} 
It is known that the ideals $\mathsf{I}_A$ generated by nonnegative regular matrices as in \eqref{eq:defnonngetaiveregular} are a proper subclass of the family of analytic $P$-ideals, 
see \cite{MR3405547, Filipow18, MR4041540} (of course, if $a_{n,k}=1/n$ for all $n\le k$ and $a_{n,k}=0$ otherwise, then $\mathsf{I}_A=\mathsf{Z}$). 
%see \cite[Proposition 13]{MR3405547} and \cite{Filipow18, MR4041540}. 
Other explicit examples of analytic P-ideals are: 
the Fubini products $\emptyset \times \mathrm{Fin}$, which can be defined as
$$
\{S\subseteq \omega: \forall n \in \omega, S \cap I_n \in \mathrm{Fin}\},
$$
where $(I_n)$ is a given partition of $\omega$ into infinite sets, 
Erd{\H o}s--Ulam ideals introduced by Just and Krawczyk in \cite{MR748847}, 
certain ideals used by Louveau and Veli\u{c}kovi\'{c} \cite{Louveau1994}, 
and, more generally, density-like ideals and generalized density ideals \cite{MR3436368, MR4404626}. Additional pathological examples can be found in \cite{MR0593624}. 
It has been suggested in \cite{MR4124855, MR3436368} that the theory of analytic $P$-ideals may have some relevant yet unexploited potential for the study of the geometry of Banach spaces. 

By a classical result of Solecki, see \cite[Theorem 3.1]{MR1708146}, an ideal $\mathsf{I}$ on $\omega$ is an analytic $P$-ideal if and only if there exists a lscsm
$\varphi$ 
%$\varphi: \mathcal{P}(\omega) \to [0,\infty]$ 
such that
\begin{equation}\label{eq:characterizationanalPideal}
\mathsf{I}=\mathrm{Exh}(\varphi):=\{S\subseteq \omega: \|S\|_\varphi=0\} 
\,\,\,\text{ and }\,\,\,
0<\|\omega\|_\varphi \le \varphi(\omega)<\infty,
\end{equation}
where 
$$
\|S\|_\varphi:=\inf\{\,\varphi(S\setminus F): F \in \mathrm{Fin}\}
$$
for all $S\subseteq \omega$. In particular, every analytic $P$-ideal is necessarily $F_{\sigma\delta}$. 
Note that $\|\cdot\|_\varphi$ is a monotone subadditive function which is invariant modulo finite sets. For instance, it is known that $\mathsf{Z}=\mathrm{Exh}(\nu)$, where
\begin{equation}\label{eq:defnuZ}
\nu(S)=\sup\nolimits_{n \in \omega}\frac{|S\cap [2^n,2^{n+1})|}{2^n}
\end{equation}
for all $S\subseteq \omega$, see e.g. \cite[Theorem 1.13.3(a)]{MR1711328} and \cite[Theorem 2]{MR4265483}.

For the next result, we need the following definition:
\begin{defi}\label{defi:stronginvariant}
An analytic $P$-ideal $\mathsf{I}$  on $\omega$ is \emph{strongly-right-translation invariant} if it is generated by a lscsm $\varphi$ as in \eqref{eq:characterizationanalPideal} such that 
%for all $S\subseteq \omega$ there exists a real constant $c=c(S)>0$ for which 
%\begin{equation}\label{eq:stronglyright}
%\|S-k\|_\varphi \ge c\,\|S\|_\varphi
%\end{equation}
%for all integers $k \in \omega$.
there exists $c>0$ such that %for all $S\subseteq \omega$ there exists a real constant $c=c(S)>0$ for which 
\begin{equation}\label{eq:stronglyright}
\|S-k\|_\varphi \ge c\,\|S\|_\varphi
\end{equation}
for all $S\subseteq \omega$ and all integers $k \in \omega$.
\end{defi}
Note that, in such case, the ideal $\mathsf{I}=\mathrm{Exh}(\varphi)$ would be right-translation invariant. Of course, if $\|\cdot\|_\varphi$ is \emph{translation invariant} (that is, $\|S\|_\varphi=\|S+k\|_\varphi=\|S-k\|_\varphi$ for all $k \in \omega$) 
then $\mathsf{I}=\mathrm{Exh}(\varphi)$ satisfies the above property. This is the case, e.g., of the ideal $\mathsf{Z}$ taking into account the definition of $\nu$ in \eqref{eq:defnuZ}. 

With these premises, we can prove the analogue of Theorem \ref{thm:Fsigmaupper} for analytic $P$-ideals. 
%, with the exception of the translation invariance property \ref{item:2upper}, (the complements of) analytic $P$-ideals provide a really large family of examples.
\begin{thm}\label{thm:analyticPidealupper}
Let $\mathsf{I}$ be an analytic $P$-ideal on $\omega$. Then $\mathsf{I}^+$ is a Furstenberg family which admits a representation as in \eqref{eq:upperFurstenberg} satisfying \ref{item:1upper} and \ref{item:3upper}. 

If, in addition, $\mathsf{I}$ is strongly-right-translation invariant then $\mathsf{I}^+$ satisfies also \ref{item:2upper}.
\end{thm}
\begin{proof}
Let $\varphi$ be a lscsm such that $\mathsf{I}=\mathrm{Exh}(\varphi)$ as in \eqref{eq:characterizationanalPideal}. Set $I=\omega$ and define 
\begin{equation}\label{def:Fijdefinition}
\mathsf{F}_{i,j}:=\left\{S\subseteq \omega: \varphi(S\setminus [0,j])> 2^{-i}\right\}
\end{equation}
for all $i,j \in \omega$. Then it is immediate to see that 
$$
\mathsf{I}^+=
\bigcup\nolimits_{i\in \omega}\left\{S\subseteq \omega: \|S\|_\varphi> 2^{-i}\right\}
=\bigcup\nolimits_{i\in \omega }\bigcap\nolimits_{j\in \omega} \mathsf{F}_{i,j}.
$$

At this point, fix $i,j \in \omega$ and pick $A \in \mathsf{F}_{i,j}$. Since $\varphi$ is a lower semicontinuous submeasure and $\|\cdot\|_\varphi$ is invariant modulo finite sets, we have 
$$
\lim_{n\to \infty} \varphi(A\cap (j,n])=\varphi(A \setminus [0,j]) \ge \|A\setminus [0,j]\|_\varphi=\|A\|_\varphi> 2^{-i}.
$$
Hence there exists a finite subset $F\subseteq A\setminus [0,j]$ such that $\varphi(F)>2^{-i}$. By the monotonicity of $\varphi$, it follows that $B \in \mathsf{F}_{i,j}$ whenever $F\subseteq B$, which proves \ref{item:1upper}.

Now, fix a set $A \in \mathsf{I}^+$ and pick an integer $i \in \omega$ such that $\|A\|_\varphi>2^{-i}$. Since $\varphi$ is a lscsm, it follows that 
$
\varphi(A\setminus [0,k]) \ge 
\lim_{n} \varphi(A\setminus [0,n])=
\|A\|_\varphi,
$ for all $k \in \omega$, hence $A\setminus [0,k] \in \mathsf{F}_{i,j}$ for all $k,j \in \omega$. Therefore property \ref{item:3upper} holds.

Lastly, suppose that $\mathsf{I}$ is strongly-right-translation invariant. %Pick $\varphi$ as in Definition \ref{defi:stronginvariant} and define $\mathsf{F}_{i,j}$ as in \eqref{def:Fijdefinition}. 
%Thanks to Theorem \ref{thm:analyticPidealupper}, we only need to show property \ref{item:2upper}. For, 
Fix $A \in \mathsf{I}^+$ so that $\|A\|_\varphi>0$. Let $c$ be the witnessing constant which satisfies inequality \eqref{eq:stronglyright} and pick $i \in \omega$ such that $2^{-i}<c\|A\|_\varphi$. It follows that 
$$
\varphi((A-k)\setminus [0,j]) \ge \|A-k\|_\varphi \ge c\|A\|_\varphi > 2^{-i},
$$
i.e., $A-k \in \mathsf{F}_{i,j}$, for all $j,k \in \omega$. This proves \ref{item:2upper} and completes the proof. 
\end{proof}

Upper Furstenberg families are useful as they allow, at first, to prove a version of the Birkhoff's transitivity theorem \cite{MR3847081}. Let us specialize the latter to the case of analytic $P$-ideals, which we report here for later reference:
\begin{prop}\label{prop:birkoff}
Let $(X,T)$ be a dynamical system where $X$ is a Baire second countable space. 
Also, let $\mathsf{I}=\mathrm{Exh}(\varphi)$ be a strongly-right-translation invariant analytic $P$-ideal. 

%, as in Definition \ref{defi:stronginvariant}. 
Then the following are equivalent\textup{:}
\begin{enumerate}[label={\rm (\roman{*})}]
\item $\mathrm{Univ}_T(\mathsf{I})$ is comeager\textup{;}
\item $T$ is $\mathsf{I}$-universal\textup{;}
\item for each nonempty open $U\subseteq X$ there exists $\kappa>0$ such that for all nonempty open sets $V \subseteq X$ there exists $y \in V$ for which 
$
\|N(y,U)\|_\varphi \ge \kappa.
$
\end{enumerate}
\end{prop}
\begin{proof}
It follows by 
%Theorem \ref{thm:univTemptycomeager}, 
Theorem \ref{thm:analyticPidealupper}, and \cite[Theorem 3.1 and Remark 3.2]{MR3847081}. 
\end{proof}

%One may wonder whether the converse of the above results hold, namely,  if $\mathsf{I}$ is an ideal on $\omega$ such that $\mathsf{I}^+$ is an upper Furstenberg family then $\mathsf{I}$ is an analytic P-ideal. \textcolor{red}{COMPLETE, case of $\mathsf{B}$ not P-ideal, mail Simone.} 

We provide below another example of an upper Furstenberg family. 
%which will be needed later. 
For, recall that $S\subseteq \omega$ is said to be \emph{syndetic} if it has bounded gaps, that is, there exists $k \ge 1$ such that $S\cap [n,n+k]\neq \emptyset$ for all $n \in \omega$. Observe that $S$ is syndetic if and only if it has positive lower Banach density, that is, $$\mathsf{bd}_\star(S):=1-\mathsf{bd}^\star(\omega\setminus S)>0.$$
\begin{example}\label{example:syndetic}
Define the family 
$$
\mathsf{S}:=\left\{A\subseteq \omega: A \text{ is syndetic}\right\}
$$ 
Note that $\mathsf{S}$ is \emph{not} the complement of an ideal on $\omega$: indeed, both $A:=\bigcup_n [2^{2n},2^{2n+1}] \cap \omega$ and $\omega\setminus A$ do not belong to $\mathsf{S}$, while their union is $\omega$. Considering that $S$ is syndetic if and only if $\mathsf{bd}_\star(S)>0$, it follows that we can represent $\mathsf{S}$ as $\bigcup\nolimits_{i,m} \bigcap\nolimits_{n,t} \mathsf{S}_{i,m,n,t}$, where 
$$
%\mathsf{G}=\bigcup\nolimits_{i \in \omega} \bigcap\nolimits_{i,j} G_{i,j}
\mathsf{S}_{i,m,n,t}:=\left\{A\subseteq \omega: |A\cap [t+m,t+m+n]| \ge 2^{-i}n\right\}
$$
for all $i,m,n,t \in \omega$. %Now, we claim that $\mathsf{S}$ is an upper Furstenberg family. %{defi:upperFurstenberg}

Of course, if $A \in \mathsf{S}_{i,m,n,t}$ and $F=[t+m,t+m+n]$ then $B \in \mathsf{S}_{i,m,n,t}$ whenever $A\cap F\subseteq B$, which implies \ref{item:1upper}. 
Now, fix $A \in \mathsf{S}$ and $k \in \omega$. Hence, by the representation of $\mathsf{S}$, there exist $i_0,m_0 \in \omega$ such that $|A\cap [t+m_0,t+m_0+n]| \ge 2^{-i_0}n$ for all $n,t \in \omega$. Without loss of generality, we can assume that $m_0 \ge k$. Considering that 
$$
(A-k) \cap [t+m_0,t+m_0+n]=\left(A \cap [k+t+m_0
,k+t+m_0+n]\right)-k,
$$
we obtain that $A-k \in \mathsf{S}_{i_0,m_0,n,t}$ for all $n,t \in \omega$, which implies \ref{item:2upper}. Therefore $\mathsf{S}$ is an upper Furstenberg family. However, the above representation of $\mathsf{S}$ does not satisfy \ref{item:3upper} (we omit details). 
%since the lower Banach density is translation invariant, then $\mathsf{bd}_\star(A)=\mathsf{bd}_\star(A-k)$ for all $A\subseteq \omega$ and $k \in \omega$, which implies \ref{item:2upper} because $\bigcap\nolimits_{n,k}\mathsf{S}_{i,n,k}=\left\{S\subseteq \omega: \mathsf{bd}_\star(S)\ge 2^{-i}\right\}$. Therefore $\mathsf{S}$ is an upper Furstenberg family. With the same reasoning, $\mathsf{S}$ satisfies also \ref{item:3upper}. 
\end{example}

Hereafter, following e.g. \cite{MR4489276}, given a dynamical system $(X,T)$, we say that a point $x \in X$ is \emph{uniformly recurrent} if $N(x,U) \in \mathsf{S}$ for all open neighborhoods $U\subseteq X$ of $x$.

%
%\begin{defi}
%\textcolor{blue}{Let $(X,T)$ be a dynamical system. Then $x \in X$ is said to be \emph{uniformly recurrent} if $N(x,U) \in \mathsf{S}$ for all open neighborhood $U\subseteq X$ of $x$.}
%\end{defi}

\section{Basic Structural and Topological Properties}\label{sec:topologicalproperties}\label{sec:structuralbasics}

We start by proving some basic properties which are mostly folklore, and whose short proofs are included for the sake of completeness, see e.g. \cite{MR3847081, MR1685272}. 

%\begin{lem}\label{lem:separability}
%Let $\mathsf{I}$ be an ideal on $\omega$ and $\mathcal{T}=(T_n: n \in \omega)$ be a $\mathsf{I}$-universal sequence of continuous maps $T_n: X\to Y$, where $X,Y$ are topological spaces. 
%Then 
%$Y$ is separable. 
%\end{lem}
%\begin{proof}
%Pick $x \in \mathrm{Univ}_{\mathcal{T}}(\mathsf{I})$. 
%Then the image of $\mathcal{T}x$ is dense in $Y$. 
%\end{proof}

%\textcolor{red}{Caso $\mathsf{I}=\mathrm{Fin}$: dove si trovano i seguenti?} 

%See \cite[Proposition 1]{MR1685272} \textcolor{red}{e successive}:

\begin{lem}\label{lem:easylem}
Let $\mathsf{I}$ be an ideal on $\omega$ and $\mathcal{T}=(T_n: n \in \omega)$ be a sequence of continuous selfmaps on a topological space $X$. 
Fix also $k \in \omega$ such that $T_nT_k=T_kT_n$ for all $n \in \omega$ and $T_k$ has dense range. 
Then 
$T_kx$ is $\mathsf{I}$-universal whenever $x$ is $\mathsf{I}$-universal.
%$T_k\left[\mathrm{Univ}_{\mathcal{T}}(\mathsf{I})\right]\subseteq \mathrm{Univ}_{\mathcal{T}}(\mathsf{I})$.  
\end{lem}
\begin{proof}
%Suppose $\mathrm{Univ}_{\mathcal{T}}(\mathsf{I})$
%is nonempty, otherwise the claim is trivial. Pick $x \in \mathrm{Univ}_{\mathcal{T}}(\mathsf{I})$ 
Pick an $\mathsf{I}$-universal point $x$ 
and a nonempty open set $U\subseteq X$. Then $V:=T_k^{-1}U$ is a nonempty open set and 
$$
%N_{\mathcal{T}}(T_kx,U)=\{n \in \omega: T_kT_nx\in U\}=N_{\mathcal{T}}(x,V) \in \mathsf{I}^+. 
N(T_kx,U)=\{n \in \omega: T_kT_nx\in U\}=N(x,V) \in \mathsf{I}^+. 
$$
Therefore $T_kx \in \mathrm{Univ}_{\mathcal{T}}(\mathsf{I})$.
\end{proof}

\begin{prop}\label{prop:denseIuniversal}
%Let $\mathsf{I}$ be an ideal on $\omega$ and $T$ be a continuous selfmap on a topological space $X$. 
Let $(X,T)$ by a dynamical system and $\mathsf{I}$ be an ideal on $\omega$. 
Then $\mathrm{Univ}_{T}(\mathsf{I})$ is either empty or dense. 
%$\mathcal{T}=(T^n: n \in \omega)$ be a sequence of iterates of a continuous map $T$ on a topological space $X$. Then $\mathrm{Univ}_{\mathcal{T}}(\mathsf{I})$ is either empty or dense. 
\end{prop}
\begin{proof}
Suppose $\mathrm{Univ}_{T}(\mathsf{I})$ is nonempty and pick a $\mathsf{I}$-universal point $x$. Since the image of its orbit is dense, then $T$ has dense image and, thanks to Lemma \ref{lem:easylem}, $Tx \in \mathrm{Univ}_{T}(\mathsf{I})$. Iterating this reasoning, we obtain $\{T^nx: n \in \omega\} \subseteq \mathrm{Univ}_{T}(\mathsf{I})$, completing the proof. 
\end{proof}

It is worth noting that the analogue of Proposition \ref{prop:denseIuniversal} does not hold for $\mathsf{I}$-recurrence. 
Indeed, let $\mathsf{I}$ be an ideal on $\omega$ and let $T$ be the bounded linear operator operator on $\ell_2$ defined by $x \mapsto (x_0,0,0,\ldots)$. Then $(1,0,0,\ldots)$ is a nonzero fixed point of $T$ (and, hence, it is an $\mathsf{I}$-recurrent point), but $T$ is clearly not $\mathsf{I}$-recurrent. 
%(i.e., $\mathrm{Rec}_T(\mathsf{I})$ is not dense). 
Therefore $\mathrm{Rec}_T(\mathsf{I})$ is neither nonempty nor dense. 
%Indeed, it has been shown in \cite[Example 2.3]{MR4489276} that there exists a linear operator $T$ on $\ell_2\times \ell_2$ with a nonzero periodic point (and, hence, $\mathsf{B}$-recurrent) such that $T$ is not $\mathsf{B}$-recurrent (i.e., $\mathrm{Rec}_T(\mathsf{B})$ is not dense). 

%An ideal $\mathsf{I}$ on $\omega$ is said to be \emph{left-translation invariant} [resp. \emph{right-translation invariant}] if $S-h\in \mathsf{I}$ [resp. $S+h\in \mathsf{I}$] for all $S \in \mathsf{I}$ and $h \in \omega$, where $S+h:=\{n+h: n \in S\}\cap \omega$ for all integers $\omega$. 
%An ideal $\mathsf{I}$ on $\omega$ is said to be \emph{right-translation invariant} if $S+k \in \mathsf{I}$ for all $S \in \mathsf{I}$ and $k \in \omega$. Equivalently, since ideals are invariant modulo finite sets, $S-k:=\{n \in\omega: n+k \in S\} \in \mathsf{I}^+$ for all $S \in \mathsf{I}^+$ and $k \in \omega$ (the definition of \emph{left-translation invariant} ideal, which will be needed later, is analogous). 
We proceed to show that, in the case of right-translation invariant ideals $\mathsf{I}$, the analogue 
of Proposition \ref{prop:denseIuniversal} 
holds for $\mathsf{I}$-strong universal points. 
\begin{prop}
Let $(X,T)$ by a dynamical system and $\mathsf{I}$ be a right-translation invariant ideal on $\omega$. 
Then $\mathrm{sUniv}_{T}(\mathsf{I})$ is either empty or dense. 
%$\mathcal{T}=(T^n: n \in \omega)$ be the sequence of iterates of a continuous map $T$ on a topological space $X$. Then $\mathrm{sUniv}_{\mathcal{T}}(\mathsf{I})$ is either empty or dense. 
\end{prop}
\begin{proof}
Suppose that $\mathrm{sUniv}_{T}(\mathsf{I})$ is nonempty, pick an $\mathsf{I}$-strong universal point $x$ (so that $\{T^nx: n \in \omega\}$ is dense), and fix an arbitrary $\eta \in X$. Then there exists $S \in \mathsf{I}^+$ such that $\lim_{n \in S}T^nx=\eta$. For each $k \in \omega$, it follows that $S-k \in \mathsf{I}^+$ and $\lim_{n \in S-k}T^n(T^kx)=\eta$. Therefore $\{T^kx: k \in \omega\}\subseteq \mathrm{sUniv}_{T}(\mathsf{I})$.
\end{proof}

We study now the topological complexity and the topological largeness of the set of $\mathsf{I}$-universal and $\mathsf{I}$-recurrent points (and their strong variants). %To this aim, an ideal $\mathsf{I}$ is a \emph{P-ideal} if it is $\sigma$-directed modulo finite sets, that is, for every sequence $(S_n)$ of sets in $\mathsf{I}$ there exists $S \in \mathsf{I}$ such that $S_n\setminus S$ is finite for all $n\in \omega$. 

%\textcolor{red}{Add ref and examples.}

\begin{thm}\label{thm:clusterpointFsigmaideal}
Let $\mathsf{I}$ be a $F_\sigma$-ideal on $\omega$ and $\mathcal{T}=(T_n: n \in \omega)$ be a sequence of continuous maps $T_n: X\to Y$, where $X$ and $Y$ are topological spaces. Then\textup{:} %with $Y$ second countable. 
%Let $\mathsf{I}$ be a $F_\sigma$-ideal on $\omega$, $X$ be a second countable space, and $\mathcal{T}=(T_n: n \in \omega)$ be a sequence of continuous selfmaps. 
%
\begin{enumerate}[label={\rm (\roman{*})}]
\item If $Y$ is second countable, then $\mathrm{Univ}_{\mathcal{T}}(\mathsf{I})$ is a $G_\delta$-set\textup{;} 
\item If $X$ is first countable, then $\mathrm{Univ}_{\mathcal{T}}(\mathsf{I})=\mathrm{sUniv}_{\mathcal{T}}(\mathsf{I})$\textup{;}
\item \label{item:3recgdelta} If $X$ is metrizable and $X=Y$, then $\mathrm{Rec}_{\mathcal{T}}(\mathsf{I})$ is a $G_\delta$-set and coincides with $\mathrm{sRec}_{\mathcal{T}}(\mathsf{I})$\textup{.}
\end{enumerate}
%If, in addition, $X$ is first countable, then $\mathrm{Univ}_{\mathcal{T}}(\mathsf{I})=\mathrm{sUniv}_{\mathcal{T}}(\mathsf{I})$. 
%If, in addition, $X=Y$, then $\mathrm{Rec}_{\mathcal{T}}(\mathsf{I})$ is also a $G_\delta$-set and coincides with $\mathrm{sRec}_{\mathcal{T}}(\mathsf{I})$.
\end{thm}
\begin{proof}
Let $\varphi$ be a lscsm such that $\mathsf{I}=\mathrm{Fin}(\varphi)$ as in \eqref{eq:mazurcharacter}.  
Also, let $(U_k)$ be a countable base of nonempty open sets of $Y$. Then 
\begin{displaymath}
\begin{split}
\mathrm{Univ}_{\mathcal{T}}(\mathsf{I})
&=\bigcap_{k\in \omega}\bigcap_{t \in \omega}\left\{x\in X: \varphi(\{n \in \omega: T_nx \in U_k\})> 2^{t}\right\}\\
&=\bigcap_{k\in \omega}\bigcap_{t \in \omega}\bigcup_{a \in \omega} 
\left\{x\in X: \varphi(\{n \le a: T_nx \in U_k\})> 2^{t}\right\}.
\end{split}
\end{displaymath}
Since the inner sets are open, it follows that $\mathrm{Univ}_{\mathcal{T}}(\mathsf{I})$ is a $G_{\delta}$-set. 

The second part follows by the fact that if $\bm{x}$ is a sequence taking values in a first countable space $X$ and $\mathsf{I}$ is a $F_\sigma$-ideal then every $\mathsf{I}$-cluster point of $\bm{x}$ is an $\mathsf{I}$-limit point (and viceversa), so that $\Gamma_{\bm{x}}(\mathsf{I})=\Lambda_{\bm{x}}(\mathsf{I})$, see \cite[Theorem 2.3]{MR3883171}; cf. also \cite[Proposition 3.6]{FKL} for a combinatorial characterization of ideals $\mathsf{I}$ satisfying the latter property. 

For the third part, let $d$ be a compatible metric on $X$, and for each $x \in X$, let $U_k$ be the open ball with center $x$ and radius $2^{-k}$. Then the proof goes verbatim as above by noting that 
$$
\left\{x\in X: \varphi(\{n \le a: d(T_nx,x)<2^{-k}\})> 2^{t}\right\}
$$
is open for all $a,t,k \in \omega$ (we omit further details). 
%in fact, if $n_1,...,n_r are given, we need to show that {x: d(T_{n_i}x,x)<2^{-k} for all i} is open. It is enough to prove for each i. If it is nonempty, pick x_0 in such set. We claim that it is an interior point. For, d(Tx_0,x_0)=2^{-k}-\delta. Then d(Tx,x) \le d(Tx,Tx0)+d(Tx_0,x_0)+d(x_0,x). The first and third term can be each smaller than delta/2 if x is sufficiently near x_0. End. 
\end{proof}

Another proof of the first claim above essentially follows by \cite[Theorem 3.1, Remark 3.2, Corollary 3.3]{MR3847081} and Theorem \ref{thm:analyticPidealupper}. 

Thanks to Proposition \ref{prop:denseIuniversal} and Theorem \ref{thm:clusterpointFsigmaideal}, we obtain: 
\begin{cor}\label{cor:Gsigmaideal}
Let $(X,T)$ be a dynamical system, where $X$ is second countable, and let $\mathsf{I}$ be a $F_\sigma$-ideal. Then the following hold\textup{:}
\begin{enumerate}[label={\rm (\roman{*})}]
\item $T$ is $\mathsf{I}$-universal if and only it is $\mathsf{I}$-strong universal\textup{;}
\item $T$ is $\mathsf{I}$-recurrent if and only if it is $\mathsf{I}$-strong recurrent\textup{;}
\item $\mathrm{Univ}_{T}(\mathsf{I})$ is either empty or a dense $G_\delta$-set\textup{.}
%\item If, in addition, $X$ is metrizable, then $\mathrm{Rec}_{T}(\mathsf{I})$ is either empty or a dense $G_\delta$-set\textup{.}
\end{enumerate}
\end{cor}

%Using also Proposition \ref{prop:denseIuniversal}, we obtain:
%\begin{cor}
%Let $(X,T)$ be a dynamical system with $X$ second countable Baire space and let $\mathsf{I}$ be a $F_\sigma$-ideal. Then both $\mathrm{Univ}_{T}(\mathsf{I})$ and $\mathrm{Rec}_{T}(\mathsf{I})$ are either empty or dense $G_\delta$-sets. 
%\end{cor}

%In particular, if $X$ is complete metric space and $\mathsf{I}$ is a $F_\sigma$-ideal then $\mathrm{Rec}_T(\mathsf{I})$ is either empty or comeager. It is worth noting that the same conclusion does not hold for all ideals $\mathsf{I}$. Indeed, it has been shown in \cite[Example 2.4]{MR4489276} that there exists a $\mathsf{B}$-recurrent linear operator on $\ell_2$ such that $\mathrm{Rec}_T(\mathsf{B})$ is meager. 

Let us proceed to the analogue of Theorem \ref{thm:clusterpointFsigmaideal} for analytic $P$-ideals:
\begin{thm}\label{thm:clusterpointanalyticPideal} 
Let $\mathsf{I}$ be an analytic P-ideal on $\omega$ and $\mathcal{T}=(T_n: n \in \omega)$ be a sequence of continuous maps $T_n: X\to Y$, where $X$ and $Y$ are topological spaces. Then\text{:}
\begin{enumerate}[label={\rm (\roman{*})}]
\item If $Y$ is second countable, then $\mathrm{Univ}_{\mathcal{T}}(\mathsf{I})$ is a $G_{\delta\sigma\delta}$-set\textup{;}
\item If $X$ is metrizable and $X=Y$, then $\mathrm{Rec}_{\mathcal{T}}(\mathsf{I})$ is a $G_{\delta\sigma\delta}$-set\textup{.}
\end{enumerate}
%Let $\mathsf{I}$ be analytic P-ideal on $\omega$, $X$ be a second countable space, and $\mathcal{T}=(T_n: n \in \omega)$ be a sequence of continuous selfmaps. 
%
%Then $\mathrm{Univ}_{\mathcal{T}}(\mathsf{I})$ is a $G_{\delta\sigma\delta}$-sets. 
%
%If, in addition, $X=Y$, then $\mathrm{Rec}_{\mathcal{T}}(\mathsf{I})$ is also a $G_{\delta\sigma\delta}$-sets. 
\end{thm}
\begin{proof}
Let $\varphi$ be a lscsm such that $\mathsf{I}=\mathrm{Exh}(\varphi)$ as in \eqref{eq:characterizationanalPideal}. 
Also, let $(U_k)$ be a countable base of nonempty open sets of $Y$. Then 
\begin{displaymath}
\begin{split}
\mathrm{Univ}_{\mathcal{T}}(\mathsf{I})
&=\bigcap_{k\in \omega}\bigcup_{t \in \omega}\left\{x\in X: \left\|\{n \in \omega: T_nx \in U_k\}\right\|_\varphi>2^{-t}\right\}\\
&=\bigcap_{k\in \omega}\bigcup_{t \in \omega}\bigcap_{a \in \omega} 
\left\{x\in X:\varphi(\{n \in [a,b]: T_nx \in U_k\})>2^{-t} \text{ for some }b\ge a\right\}.
\end{split}
\end{displaymath}
Since the inner sets are open, it follows that $\mathrm{Univ}_{\mathcal{T}}(\mathsf{I})$ is a $G_{\delta\sigma\delta}$-set. 

The second part goes verbatim, as in the proof of Theorem \ref{thm:clusterpointFsigmaideal}\ref{item:3recgdelta}.
%letting $(U_k)$ be a decreasing local base at each $x$. 
\end{proof}

With the same reasoning of Corollary \ref{cor:Gsigmaideal}, we conclude that the set of $\mathsf{I}$-universal points is either empty or a dense $G_{\delta\sigma\delta}$-set. 

In the case where $\mathsf{I}$ is also strongly-right-translation invariant, we can show that it contains a dense $G_\delta$-set, whenever it is nonempty. 
This generalizes the known result that the set of $\mathsf{Z}$-universal vectors of linear dynamical systems, where $X$ is a Banach space, is either empty or comeager, see e.g. \cite[Proposition 21]{MR3334899}. %\textcolor{red}{ADD OTHERS}
\begin{thm}\label{thm:univTemptycomeager}
Let $(X,T)$ be a dynamical system, where $X$ is second countable, and let $\mathsf{I}$ be a strongly-right-translation invariant analytic P-ideal on $\omega$. 
%If $T$ admits a $\mathsf{I}$-universal point, then $\mathrm{Univ}_{T}(\mathsf{I})$ is dense and comeager. 

Then $\mathrm{Univ}_{T}(\mathsf{I})$ is either empty or contains a dense $G_\delta$-set. 
%Let $\mathsf{I}$ be an analytic P-ideal on $\omega$ which is strongly-right-translation invariant and $\mathcal{T}=(T_n: n \in \omega)$ be a sequence of continuous maps $T_n: X\to Y$, where $X$ and $Y$ are topological space with $Y$ second countable. 
%Let $\mathsf{I}$ be analytic P-ideal on $\omega$, $X$ be a second countable space, and $\mathcal{T}=(T_n: n \in \omega)$ be a sequence of continuous selfmaps. 
%
%If $\mathcal{T}$ admits a $\mathsf{I}$-universal point, then $\mathrm{Univ}_{\mathcal{T}}(\mathsf{I})$ is dense and comeager. 
\end{thm}
\begin{proof}
Let $\varphi$ be a lscsm for which $\mathsf{I}=\mathrm{Exh}(\varphi)$ and pick $c$ as in Definition \ref{defi:stronginvariant}. Let also $(U_k)$ be as in the proof of Theorem \ref{thm:clusterpointanalyticPideal} and  suppose that $T$ admits an $\mathsf{I}$-universal point $x_0 \in X$. For each $k \in \omega$, define the positive real
$$
\varepsilon_k\coloneqq 
\frac{c}{2}\,
\|\{n \in \omega: T^nx_0 \in U_k\}\|_\varphi.
$$
It is clear from 
%the representation of the $\mathsf{I}$-universal points in 
the proof of Theorem \ref{thm:clusterpointanalyticPideal} that $\mathrm{Univ}_{T}(\mathsf{I})$ contains 
$$
X_0:=\bigcap_{k\in \omega}\bigcap_{a\in \omega} \left\{x\in X:\varphi(\{n \in [a,b]: T^nx \in U_k\})>\varepsilon_k \text{ for some }b\ge a\right\}.
$$
Note that the inner sets are open, hence $X_0$ is a $G_\delta$-set. 
Lastly, observe that $T^mx_0 \in X_0$ for all $m \in \omega$. Indeed, since $\varphi$ satisfies \eqref{eq:stronglyright}, we obtain 
\begin{displaymath}
\begin{split}
\lim_{b\to \infty} \varphi(\{n \in [a,b]: T^n(T^mx_0) \in U_k\})
%&=\lim_{b\to \infty} \varphi(\{n \in [m+a,m+b]: T^nx_\star \in U_k\})\\
&=\varphi(\{n \ge a: T^{n+m}x_0 \in U_k\})\\
&\ge \|\{n \in \omega: T^nx_0 \in U_k\}-m\,\|_\varphi \\
&\ge c\|\{n \in \omega: T^nx_0 \in U_k\}\,\|_\varphi
>\varepsilon_k
\end{split}
\end{displaymath}
for all $a,m,k \in \omega$. 
The lower semicontinuity of $\varphi$ implies that 
$$
\{T^mx_0: m \in \omega\}\subseteq X_0\subseteq \mathrm{Univ}_T(\mathsf{I}), 
$$
therefore $X_0$ is a dense $G_\delta$-subset of $\mathsf{I}$-universal points.
\end{proof}

We present two simple consequences below. For, recall that a topological space $X$ is Baire if countable unions of closed subsets with empty interior also have empty interior (examples include compact Hausdorff spaces, complete metric spaces, algebraic varieties with the Zariski topology, etc.). The following is another folklore result. %which we include for the sake of completeness.
\begin{lem}\label{lem:sumsetcomeager}
Let $X$ be a Baire topological group and $S\subseteq X$ comeager. 
%Let $S$ be a comeager subset of a Baire topological group $X$. 
Then $S+S=X$.
\end{lem}
\begin{proof}
Fix $x \in X$. Then $x-S$ is comeager, hence also $S\cap (x-S)$. Considering that $X$ is Baire, then $S\cap (x-S)$ is nonempty. 
\end{proof}

A strong version of \cite[Proposition 5]{MR1685272} follows, cf. also \cite[Proposition 1.29]{MR2533318}:
\begin{cor}\label{cor:sumset}
Let $(X,T)$ be a dynamical system, where $X$ is a second countable Baire topological group, and let $\mathsf{I}$ be a $F_\sigma$-ideal on $\omega$ for which $T$ is $\mathsf{I}$-universal. Then every $x \in X$ can be written as sum of two $\mathsf{I}$-strong universal points. 
\end{cor}
\begin{proof}
It follows by Proposition \ref{prop:denseIuniversal}, Theorem \ref{thm:clusterpointFsigmaideal}, and Lemma \ref{lem:sumsetcomeager}.
\end{proof}

Analogously to Corollary \ref{cor:sumset}, we have:

\begin{cor}\label{cor:sumsetX}
Let $(X,T)$ be a dynamical system, where $X$ is a second countable Baire topological group. 
Let $\mathsf{I}$ be a strongly-right-translation invariant analytic P-ideal such that $T$ is $\mathsf{I}$-universal. Then every $x\in X$ can be written as sum of two $\mathsf{I}$-universal points. 
\end{cor}
\begin{proof}
It follows by Theorem \ref{thm:univTemptycomeager} and Lemma \ref{lem:sumsetcomeager}.
\end{proof}

We provide now a necessary structural condition on the ideal $\mathsf{I}$ for $\mathsf{I}$-universal Lipschitz dynamical systems with a fixed point. The linear case, where $X$ is a nonzero Banach space, can be found in \cite[Proposition 1]{MR3552249}. %, which corresponds to the case where $X$ is nonzero Banach space and $T\in \mathscr{L}(X)$. 
%https://arxiv.org/pdf/2212.03652.pdf
\begin{prop}\label{prop:lipmaps}
Let $(X,T)$ be a dynamical system, where $X$ is an unbounded metric space and $T$ is Lipschitz. Also, let $\mathsf{I}$ be an ideal on $\omega$ and 
%Let $T$ be a Lipschitz continuous selfmap on an unbounded metric space $X$ and $\mathsf{I}$ be an ideal on $\omega$. Also, 
suppose that $T$ is $\mathsf{I}$-universal and admits a fixed point. 
%Then, for each function $f:\omega\to \omega$, 
Then 
there exists a sequence $(S_k: k \in \omega)$ of pairwise disjoint sets in $\mathsf{I}^+$ such that if $j \in S_k$ and $j^\prime \in S_{k^\prime}$ with $j\neq j^\prime$ then 
%$$
%|j-j^\prime| \ge \max\{f(k),f(k^\prime)\}.
%$$ 
$
|j-j^\prime| \ge \max\{k,k^\prime\}.
$
\end{prop}
\begin{proof}
Let $x$ be an $\mathsf{I}$-universal point and $x_\star$ be a fixed point of $T$. Also, denote by $d$ the metric on $X$ and let $C$ be a constant for which $d(Tx^\prime,Tx^{\prime\prime}) \le Cd(x^\prime,x^{\prime\prime})$ for all $x^\prime, x^{\prime\prime} \in X$. Note that $C>1$ since the orbit of $x$ is dense, $d$ is unbounded, and 
\begin{equation}\label{eq:inequalitytriangularlip}
\forall n,m \in \omega, \quad d(T^{n+m} x, x_\star) \le C^n d(T^mx,x_\star). 
\end{equation}

\begin{claim} \label{claim:claimsequencepropertyhypercyclic}
Given integers $p,r \ge 1$, there exist $y \in X$ and an integer $q\ge 1$ such that, if $A:=N(x,B(x_\star,p))$ and $B:=N(x,B(y,1/q))$ then $|a-b| \ge r$ and $|b-b^\prime| \ge r$ for all $a \in A$ and all distinct $b,b^\prime \in B$. 
\end{claim}
\begin{proof}
Fix an integer $h\ge 4r$ such that $d(T^hx,x_\star)\ge 4C^{4r}rp$, which exists because the orbit of $x$ is dense and $d$ is unbounded, and define 
$
y:=T^{h-2r}x.
$ 
It follows by \eqref{eq:inequalitytriangularlip} that 
$$
d(y,x_\star)=d(T^{h-2r}x,x_\star)\ge C^{-2r}d(T^h x,x_\star) \ge 4C^{2r}rp.
$$
%$$
%d(T^{h-i}x,x_\star) \ge 
%C^{-i}d(T^h x,x_\star) \ge 
%C^{-4r}\cdot d(T^h x,x_\star)\ge p
%$$
%for all $i=0,1,\ldots,4r$. In particular, by the first inequality above, $d(y,x_\star)\ge 4C^{2r}rpq$. 
Given $a \in A$, $b \in B$ and $i \in \{1,\ldots,r\}$, we obtain that 
\begin{displaymath}
\begin{split}
d(T^{a+i}x,y) &\ge d(y,x_\star)-d(T^{a+i}x,x_\star) \\
&\ge 4C^{2r}rp -C^id(T^ax,x_\star) \\
&\ge 4C^{2r}rp- C^rp \ge 4rp-p >1,
\end{split}
\end{displaymath}
so that $a+i \notin B$, and similarly 
%$d(T^{b+i}x,x_\star) \ge p$ (we omit details), 
\begin{displaymath}
\begin{split}
d(T^{b+i}x,x_\star) &\ge d(y,x_\star)-d(T^{b+i}x,y) \\
&\ge 4C^{2r}rp -C^id(T^bx,y) \\
&\ge 4C^{2r}rp- C^r 
\ge 4rp-1 > p,
\end{split}
\end{displaymath}
so that $b+i \notin A$. Lastly, since $y$ belongs to the orbit of $x$ and the metric is unbounded, it cannot be periodic. Therefore $\varepsilon:=\min\{d(T^jy,y): j=1,\ldots,r\}$ is strictly positive. Pick an integer $q>(C^r+1)/\varepsilon$. Then also
\begin{displaymath}
\begin{split}
d(T^{b+i}x,y)& \ge d(T^iy,y)-d(T^{b+i}y,T^iy) \\
&\ge \varepsilon - C^i d(T^bx,y) > \varepsilon - C^r/q > 1/q,
\end{split}
\end{displaymath}
which proves that $b+i\notin B$, and concludes the proof of the claim. 
\end{proof}

By a repeated application of Claim \ref{claim:claimsequencepropertyhypercyclic}, we can construct sequences of subsets $(A_n)$ and $(B_n)$ of $\omega$ such that $A_n\cap B_n=\emptyset$, $A_n\cup B_n\subseteq A_{n+1}$, $|a-b|\ge n$, and $|b-b^\prime| \ge n$ for all $n\in\omega$, $a \in A_n$ and $b,b^\prime \in B_n$. Therefore, the sequence $(B_n)$ has values in $\mathsf{I}^+$ (since $x$ is $\mathsf{I}$-universal) and satisfies the required property.
\end{proof}

As anticipated in the Introduction, this does not allow to show that if a certain dynamical system is $\mathsf{I}$-universal then $\mathsf{I}$ has a low topological complexity:
\begin{prop}\label{prop:nonborel}. 
Let $\mathcal{T}=(T_n: n \in \omega)$ be a sequence of continuous maps $T_n: X\to Y$, where $X$ and $Y$ are topological spaces. Suppose that $\mathcal{T}$ is $\mathsf{I}$-universal for some ideal $\mathsf{I}$ on $\omega$ with $\mathsf{I}\neq \mathrm{Fin}$. Then $\mathcal{T}$ is $\mathsf{J}$-universal for some ideal $\mathsf{J}$ which is not analytic.
\end{prop}
\begin{proof}
Pick an infinite set $I \in \mathsf{I}$ with increasing enumeration $(i_n: n \in \omega)$. Fix a maximal ideal $\mathsf{I}_0$ and define the ideal 
$$
\mathsf{J}:=\{S\in \mathsf{I}: \{n \in \omega: i_n \in S\} \in \mathsf{I}_0\}. 
$$
%(Recall that $\mathsf{I}_0$ is not analytic.) 
Note that $\mathcal{T}$ is $\mathsf{J}$-universal since $\mathsf{J}\subseteq \mathsf{I}$. Suppose for the sake of contradiction that $\mathsf{J}$ is analytic. Since $\mathcal{P}(I)$ is closed, then also $\mathsf{J}\cap \mathcal{P}(I)$ is analytic. 
%Let $\pi: \{0,1\}^\omega \to \{0,1\}^\omega$ be the embedding defined by $(x_n: n \in \omega)\mapsto (x_{i_n}: n \in \omega)$. 
However, $\mathsf{I}_0$ would be the continuous image of $\mathsf{J}\cap \mathcal{P}(I)$ through the map $\pi: \{0,1\}^\omega \to \{0,1\}^\omega$ defined by $(x_n: n \in \omega)\mapsto (x_{i_n}: n \in \omega)$. This contradicts the fact that $\mathsf{I}_0$ is not analytic. 
\end{proof}

\section{Return sets}\label{sec:returnsets}

%\textcolor{red}{Generalization of \cite[Theorem 14]{MR3552249}, già scritto sotto:}
In the Propositions of this Section, we are going to show several structural properties of return sets. As consequences, we are able to characterize the $\mathsf{I}$-universality of certain dynamical systems. %and to prove that $\lim_{n \in A}T^nx=x$ implies that $A$ contains arbitrarily large gaps whenever $x$ is not periodic. 

%We are going to show in the next results called .........

\begin{prop}\label{prop:translations}
Let $(X,T)$ be a dynamical system.   
Fix a point $x \in X$ and a nonempty open set $U\subseteq X$. 
Then, for every $y \in \overline{\mathrm{orb}(x,T)}$ and every finite $S\subseteq N(y,U)$, there exists $k \in \omega$ such that $S+k\subseteq N(x,U)$. 
\end{prop}
\begin{proof}
Fix $y \in X$ and a finite set $S\subseteq N(y,U)$. Assume that $S\neq \emptyset$ and note that $V:=\bigcap_{n \in S}T^{-n}[U]$ is an open set and, in addition, it is nonempty since $T^ny \in U$ for all $n \in S$. Since $y \in \overline{\mathrm{orb}(x,T)}$, we can pick $k \in N(x,V)$. Hence $T^kx \in V$, which implies that $T^{n+k}x \in U$ for all $n \in S$. Therefore we obtain $S+k \subseteq N(x,U)$. 
\end{proof}

The following corollary is immediate. (Note that $\mathsf{bd}^\star(S)=1$ if and only if $S$ contains arbitrarily large intervals.)
%For, recall that the upper Banach density  and lower Banach density of a set $S\subseteq \omega$ are defined by
%$$
%\mathsf{bd}^\star(S):=\lim_{n\to \infty} \max_{k \ge 0} \frac{|S\cap [k+1,k+n]|}{n}
%$$
%and
%$$
%%\,\text{ and }\,
%\mathsf{bd}_\star(S):=\lim_{n\to \infty} \min_{k \ge 0} \frac{|S\cap [k+1,k+n]|}{n},
%$$
%respectively, cf. \cite[Section 2]{MR4054777}; hence, $S$ has upper Banach density $1$ if and only if it contains arbitrary large intervals. 
\begin{cor}\label{cor:upperBanachdensity1}
Let $(X,T)$ be a dynamical system. Pick also a universal point $x$ and an open set $U\subseteq X$ containing a fixed point of $T$. Then 
%$\mathsf{bd}^\star(N(x,U))=1$, i.e., 
$N(x,U)$ contains arbitrarily large intervals. 
\end{cor}
\begin{proof}
Let $y$ be a fixed point of $T$ contained in $U$. Then $N(y,U)=\omega$ and $y \in U\subseteq X=\overline{\mathrm{orb}(x,T)}$. The claim follows by Proposition \ref{prop:translations}. 
\end{proof}

As an application of Proposition \ref{prop:translations}, we characterize $\mathsf{B}$-universality extending \cite[Theorem 14]{MR3552249} and \cite[Theorem 2.1]{MR4489276}. 
% (note that, without further hypotheses on the underlying space $X$, Baire's category theorem cannot be applied): 
\begin{thm}\label{thm:reiterativeequalhypercyclic}
%Let $T$ be a $\mathsf{B}$-universal continuous selfmap on a topological space $X$. 
Let $(X,T)$ be a dynamical system. Then the following are equivalent\textup{:}
\begin{enumerate}[label={\rm (\roman{*})}]
\item \label{item:1Buniv} $T$ is universal and $\mathrm{Univ}_T(\mathsf{B})=\mathrm{Univ}_T(\mathrm{Fin})$\textup{;}
\item \label{item:2Buniv} $T$ is $\mathsf{B}$-universal\textup{;}
\item \label{item:3Buniv} $T$ is admits a universal $\mathsf{B}$-recurrent point\textup{;}
\item \label{item:4Buniv} $T$ is universal and $\mathsf{B}$-recurrent\textup{.}
\end{enumerate}
In addition, if $X$ is a second countable Baire space, they are also equivalent to\textup{:}
\begin{enumerate}[label={\rm (\roman{*})}]
\setcounter{enumi}{4}
\item \label{item:5Buniv} $\mathrm{Univ}_T(\mathsf{B})$ is comeager\textup{;}
\item \label{item:6Buniv} the set of universal $\mathsf{B}$-recurrent points is comeager\textup{;}
\item \label{item:7Buniv} the set of universal $\mathsf{B}$-recurrent points is not meager\textup{.}
%\item \label{item:8Buniv} there exists a universal $\mathsf{B}$-recurrent point\textup{.}
\end{enumerate}
\end{thm}
\begin{proof}
The implications \ref{item:1Buniv} $\implies$ \ref{item:2Buniv} $\implies$ \ref{item:3Buniv} are trivial. 
The implication \ref{item:3Buniv} $\implies$ \ref{item:4Buniv} follows by the fact that each element of the orbit of a universal $\mathsf{B}$-recurrent point (which is dense in $X$) is necessarily $\mathsf{B}$-recurrent. 
Lastly, let us show \ref{item:4Buniv} $\implies$ \ref{item:1Buniv}. Of course, every $\mathsf{B}$-universal point is universal. On the other hand, pick a universal point $x \in X$ and a nonempty open set $U\subseteq X$. Since $T$ is $\mathsf{B}$-recurrent there exists a $\mathsf{B}$-recurrent point $y\in U$. It follows by Proposition \ref{prop:translations} that the set $N(x,U)$ contains a translation of every finite subset of $N(y,U)$, so that $\mathsf{bd}^\star(N(x,U))\ge \mathsf{bd}^\star(N(y,U))>0$. Therefore $x$ is $\mathsf{B}$-universal, concluding the proof of the first part. 

\medskip

Suppose now that $X$ is a Baire space. Note that the implications \ref{item:5Buniv} $\implies$ \ref{item:6Buniv} $\implies$ \ref{item:7Buniv} 
%$\implies$ \ref{item:8Buniv} 
$\implies$ \ref{item:3Buniv} are trivial, and the latter is equivalent to \ref{item:1Buniv}. Therefore it is sufficient to show that \ref{item:1Buniv} $\implies$ \ref{item:5Buniv}. The conclusion follows by the fact, if $T$ is universal, then $\mathrm{Univ}_T(\mathrm{Fin})$ is comeager by Corollary \ref{cor:Gsigmaideal}. 
\end{proof}

In particular, we recover a result which has been obtained by Menet \cite{MR3632557} for linear dynamical systems on Fr\'{e}chet spaces.
\begin{cor}\label{cor:chaoticimpliesBuniversal}
Let $(X,T)$ be dynamical system which is chaotic \textup{(}that is, $T$ is universal and admits a dense set of periodic points\textup{)}. Then $T$ is $\mathsf{B}$-universal. 
\end{cor}
\begin{proof}
Since each periodic points is $\mathsf{B}$-recurrent, the conclusion follows by the equivalence \ref{item:2Buniv} $\Longleftrightarrow$ \ref{item:4Buniv} in Theorem \ref{thm:reiterativeequalhypercyclic}. 
\end{proof}

%It is worth to remark that, under additional hypotheses on the dynamical system (namely, the existence of a dense set of points with null orbit), the above conclusion can be strenghtened, see Theorem \ref{thm:strenghten} below.

%\textcolor{red}{COMMENT}

We provide now a result similar to Proposition \ref{prop:translations}, which allows to show that the larger $N(U,V)$ contains arbitrarily large intervals. 
\begin{prop}\label{prop:weakmixing}
Let $(X,T)$ be a universal dynamical system. 
%, where $X$ is Hausdorff. 
Pick also nonempty open sets $U,V\subseteq X$ and a finite set $S\subseteq \omega$ such that $0 \in S$ and $\bigcap_{n \in S}T^{-n}[U]\neq \emptyset$. Then there are infinitely many $k\ge \max S$ such that $k-S \subseteq N(U,V)$.
\end{prop}
\begin{proof}
Let $\tilde{U}$ be the nonempty open set $\bigcap_{n \in S}T^{-n}[U]$ and note that $\tilde{U}\subseteq U$. Since the set of universal points is dense by Proposition \ref{prop:denseIuniversal}, we can pick a universal point $x \in \tilde{U}\subseteq U$. Fix $k$ in the infinite set $N(x,V)\setminus [0,\max S)$. 
%\ge \max S$ such that $T^kx \in V$. 
To conclude the proof, observe that $T^{k-n}(T^nx)=T^kx\in V$ and $T^nx \in U$ for all $n \in S$. Therefore $k-S \subseteq N(U,V)$. 
\end{proof}

\begin{cor}\label{cor:mixingjshfkghj}
Let $(X,T)$ be a universal dynamical system. %, where $X$ is Hausdorff. 
Pick also nonempty open sets $U,V\subseteq X$ such that $U$ contains a fixed point of $T$. Then $N(U,V)$ contains arbitrarily large intervals. 
\end{cor}
\begin{proof}
Let $x_0 \in U$ be a fixed point of $T$ and define $S:=\{0,1,\ldots,n\}$, for some $n \in \omega$. It follows by Proposition \ref{prop:weakmixing} that $N(U,V)$ contains a translation of $S$. 
\end{proof}

As another structural result, we show that $N(x,U)$ contains a translation of $N(x,V)$ for some neighborhood $V$ of $x$, whenever $U$ intersects the orbit of $x$. 
\begin{prop}\label{prop:returnleftinvariant}
Let $(X,T)$ be a dynamical system. Fix a point $x \in X$ and an open set $U\subseteq X$ such that $U\cap \mathrm{orb}(x,T)\neq \emptyset$. Then there exist $k \in \omega$ and an open neighborhood $V$ of $x$ such that $N(x,V)+k\subseteq N(x,U)$.
\end{prop}
\begin{proof}
By hypothesis, there exists $k \in \omega$ such that $T^kx \in U$. Since $T$ is continuous, there exists an open neighborhood $V$ of $x$ such that $T^k[V]\subseteq U$. It follows that, if $n \in N(x,V)$, then $T^{n+k}x=T^k(T^nx) \in U$. %Therefore $N(x,V)+k\subseteq N(x,U)$. 
\end{proof}

An analogue of Theorem \ref{thm:reiterativeequalhypercyclic} holds for left-translation invariant ideals $\mathsf{I}$. 
\begin{thm}\label{thm:upperfrequentcharacterization}
Let $(X,T)$ be a dynamical system and $\mathsf{I}$ be a left-translation invariant ideal on $\omega$. Then the following are equivalent\textup{:}
\begin{enumerate}[label={\rm (\roman{*})}]
\item \label{item:1Uuniv} there exists a universal $\mathsf{I}$-recurrent point and $\mathrm{Univ}_T(\mathrm{Fin})\cap \mathrm{Rec}_T(\mathsf{I})=\mathrm{Univ}_T(\mathsf{I})$\textup{;} 
\item \label{item:2Uuniv} $T$ is $\mathsf{I}$-universal\textup{;}
\item \label{item:3Uuniv} there exists a universal $\mathsf{I}$-recurrent point\textup{.}
\end{enumerate}
If, in addition, $X$ is a second countable Baire space and $\mathsf{I}$ is a $F_\sigma$ ideal or a strongly-right-translation invariant analytic $P$-ideal, they are also equivalent to\textup{:}
\begin{enumerate}[label={\rm (\roman{*})}]
\setcounter{enumi}{3}
\item \label{item:4Uuniv} $\mathrm{Univ}_T(\mathsf{I})$ is comeager\textup{;}
\item \label{item:5Uuniv} the set of universal $\mathsf{I}$-recurrent points is comeager\textup{;}
\item \label{item:6Uuniv} the set of universal $\mathsf{I}$-recurrent points is not meager\textup{.}
\end{enumerate}
\end{thm}
\begin{proof}
The implications \ref{item:1Uuniv} $\implies$ \ref{item:2Uuniv} $\implies$ \ref{item:3Uuniv} are clear, hence it is enough to prove \ref{item:3Uuniv} $\implies$ \ref{item:1Uuniv}. For, pick a universal $\mathsf{I}$-recurrent point $x \in X$ and an open set $U\subseteq X$. It follows by Proposition \ref{prop:returnleftinvariant} that there exist $k\in \omega$ and an open neighborhood $V$ of $x$ such that $N(x,V)+k\subseteq N(x,U)$. Since $x$ is $\mathsf{I}$-recurrent and the ideal $\mathsf{I}$ is left-invariant, we conclude that $N(x,U) \in \mathsf{I}^+$, which concludes the first part. 

\medskip

Suppose now that $X$ is a Baire space and $\mathsf{I}$ is a $F_\sigma$ ideal or an analytic $P$-ideal. Note that the implications \ref{item:4Uuniv} $\implies$ \ref{item:5Uuniv} $\implies$ \ref{item:6Uuniv} $\implies$ \ref{item:3Uuniv} are trivial, and the latter is equivalent to \ref{item:2Uuniv}. Therefore it is sufficient to show that \ref{item:2Uuniv} $\implies$ \ref{item:4Uuniv}. The conclusion follows by Corollary \ref{cor:Gsigmaideal} and Theorem \ref{thm:univTemptycomeager}.
\end{proof}

Note that, in the general framework of Theorem \ref{thm:upperfrequentcharacterization}, we are missing an equivalent condition of the type \textquotedblleft $T$ is universal and $\mathsf{I}$-recurrent.\textquotedblright\, However, even if this is correct for the ideal $\mathsf{B}$ by Theorem \ref{thm:reiterativeequalhypercyclic}, this is false for the ideal $\mathsf{Z}$: indeed, it has been shown by Menet in \cite[Theorem 1.2]{MR3632557} there exists a hypercyclic linear operator on $c_0$ which has a dense set of periodic points 
%(hence, it is $\mathsf{Z}$-recurrent) 
and, at the same time, it is not $\mathsf{Z}$-hypercyclic. 

% \textcolor{red}{COMMENTI SU MENET}

We prove now that certain return sets cannot be syndetic, extending \cite[Proposition 2]{MR3552249} and \cite[Theorem 2.10]{MR4489276} (which already contain the argument of the proof): 
%it has zero lower Banach density.) 

%
%At this point, define the gap function $\mathrm{gap}: \mathcal{P}(\omega) \to [0,\infty]$ by 
%$$
%\forall S\subseteq \omega, \quad 
%\mathrm{gap}(S):=\sup\{|I|: I\subseteq \omega \text{ interval such that }I\cap S=\emptyset\}.
%$$
%Note that, in particular, $\mathrm{gap}(\emptyset)=\infty$ and $g(S)=0$ if and only if $S=\omega$. A set $S\subseteq \omega$ which has bounded gaps (i.e., $\mathrm{gap}(S)<\infty$) is said to be \emph{syndetic}. 

%\textcolor{red}{Generalization of \cite[Proposition 2]{MR3552249}:}

%\textcolor{blue}{POWER BOUNDEDNESS: We already recalled Furstenberg’s result which says that the closure of the orbit of a uniformly recurrent vector is a minimal set. The dynamics on minimal compact sets (like irrational rotations on the torus) is a matter of study in non-linear dynamics.}
\begin{prop}\label{prop:infinitegap}
Let $(X,T)$ be a dynamical system. %, with $X$ regular Hausdorff. 
%Let $T$ be a continuous selfmap on a regular Hausdorff space $X$. 
Pick $x \in X$, a fixed point $x_0$ of $T$, and suppose that $x_0 \in \overline{\mathrm{orb}(x,T)}$. 
Then 
$N(x,U)$ is not syndetic 
%$$
%\mathrm{gap}\hspace{.3mm}(\hspace{.3mm}\mathcal{N}_f(x,U))=\infty
%$$
for every open set $U\subseteq X$ such that $x_0\notin \overline{U}$.  
\end{prop}
%REGULAR HAUSDORFF= closed and points can be separated by disjoint open neighborhoods. 
\begin{proof}
Let $U\subseteq X$ is be an open set such that $x_0\notin \overline{U}$. Assume that $N(x,U)$ is an infinite set (hence, in particular, $U\neq \emptyset$), otherwise the claim is trivial. 
Suppose for the sake of contradiction that there exists $g \in \omega$ such that $N(x,U) \cap [n,n+g]\neq \emptyset$ for all $n \in \omega$. Without loss of generality, $g$ is the minimal integer with this property. 
Since 
%$X$ is regular Hausdorff and 
$x_0\notin \overline{U}$, we can pick an open neighborhood $V$ of $x_0$ which is disjoint from $U$. 
%By hypothesis we know that $\mathcal{N}_f(x,W)\neq\emptyset$ for every nonempty open set $W$ such that $x_0 \in W\subseteq V$. 
At this point, define 
$$
W:=V \cap T^{-1}[V] \cap \cdots \cap T^{-g}[V],
$$
which is an open set such that $x_0 \in W\subseteq V$. Observe that $N(x,W)\neq \emptyset$ since $x_0 \in \overline{\mathrm{orb}(x,T)}$. 
It follows that $T^{n+k}x=T^k(T^nx) \in V$ for all $k \in \{0,\ldots,g\}$ and $n \in N(x,W)$. Therefore $N(x,U) \cap [n,n+g]=\emptyset$, which contradicts the standing hypothesis. 
\end{proof}

%Note that, since $\mathsf{bd}_\star(S)=1-\mathsf{bd}^\star(S^c)$ for all $S\subseteq \omega$, then $S$ has lower Banach density $0$ if and only if $S^c$ contains arbitrary large intervals, i.e., if and only if $S$ is not syndetic. 
\begin{cor}\label{cor:lowerBanachdensity}
Let $(X,T)$ be a dynamical system. 
%, with $X$ regular Hausdorff. 
Let also $U\subseteq X$ be a nonempty open set such that $\overline{U}$ does not contain all fixed points of $T$. 
Then $N(x,U)$ is not syndetic for all universal points $x \in X$.
\end{cor}
\begin{proof}
%Let $x \in X$ be a universal point, $U\subseteq X$ be an open set, and $x_0 \in X$ be a fixed point of $T$ such that $\mathrm{Int}(U^c)$, i.e. $x_0\notin \mathrm{Cl}(U)$. 
Since $\mathrm{orb}(x,T)$ is dense in $X$, the claim follows by Proposition \ref{prop:infinitegap}.  
\end{proof}

The following consequence should be compared with the main result obtained by Grivaux in \cite{MR3771261}, where the analogue problem has been studied for certain linear dynamical systems on Banach spaces, replacing the upper and lower Banach density with the upper and lower asymptotic density, respectively. 
\begin{cor}\label{cor:differentdensities}
Let $(X,T)$ be a dynamical system. %, with $X$ regular Hausdorff. 
Pick $x \in X$, a fixed point $x_0$ of $T$, and a nonempty open set $U\subseteq X$ with $x_0 \notin \overline{U}$. Suppose also that $U$ contains a $\mathsf{B}$-recurrent point $x_1$ and that $x_0,x_1 \in \overline{\mathrm{orb}(x,T)}$. Then %(e.g., a periodic point). Then
$$
\mathsf{bd}_\star(N(x,U))=0 
\quad \text{ and }\quad 
\mathsf{bd}^\star(N(x,U))>0.
$$
\end{cor}
\begin{proof}
The return set $N(x,U)$ contains arbitrary large gaps by Proposition \ref{prop:infinitegap}, hence $\mathsf{bd}_\star(N(x,U))=0$. In addition, since $x_1$ is $\mathsf{B}$-recurrent and belongs to $U$, for every finite subset $S\subseteq N(x_1,U)$ there exists $k \in \omega$ such that $S+k\subseteq N(x,U)$ by Proposition \ref{prop:translations}. Therefore $\mathsf{bd}^\star(N(x,U))\ge \mathsf{bd}^\star(N(x_1,U))>0$.
\end{proof}

%\textcolor{red}{COROLLARY: Let $T$ be a continuous selfmap on a regular Hausdorff space $X$ and suppose that $T$ is chaotic. Then $N_T(x,U)$ has lower Banach density $0$ for every universal point $x$ and every nonempty open set $U\subseteq X$. (DEVI POTER APPLICARE ANSARI THEOREM, E' SUFFICIENTE X TVS e T in L(X). In particular, chaotic operators cannot be frequently hyperyclic (ma è piu' forte).}

The following result has been used, more or less explicitly, in several works, including for instance \cite{MR2323544, MR4242546, Grivaux2023, MR2153969, MR3742548}. 
\begin{prop}\label{prop:differencesets}
%Let $T$ be a continuous selfmap on a topological space $X$.
Let $(X,T)$ be a dynamical system. 
Then, for all nonempty open sets $U,V \subseteq X$, and all integers 
%$n \in N(U,V):=\{n \in \omega: U\cap T_n^{-1}(V) \neq \emptyset\}$, 
$n \in N(U,V):=\{n \in \omega: T^ny\in V \text{ for some }y \in U\}$, 
there exists a nonempty open set $W\subseteq X$ such that 
$$
N(x,W)-N(x,W)+n \subseteq N(U,V)
$$
for all universal points $x\in X$. 
\end{prop}
\begin{proof}
Suppose that $x \in X$ is universal. Fix a nonempty open sets $U,V\subseteq X$, an integer $n\in N(U,V)$, and define $W:=U\cap T^{-n}(V)$. Note that the latter is open and nonempty by the universality of $x$. 
% (\textcolor{red}{To add: universal implies topological transitive so that $W\neq \emptyset$}.) 
Now, pick $p,q \in N(x,W)$ with $p>q$. Then we need to show that there exists $y \in U$ such that $T^{p-q+n}y \in V$. To this aim, set $y:=T^qx$, so that $y \in W\subseteq U$, and note that $T^p x \in W\subseteq T^{-n}(V)$. Therefore $T^{p-q+n}y=T^{p+n}x \in V$. 
% therefore
%\begin{displaymath}
%\begin{split}
%T^{p-q+n}y=T^{p+n}x \in V
%\end{split}
%\end{displaymath}
\end{proof}

%$$
%T^{2p-2q+n}y=T^{p+p-q+n}y=T^{p+p+n}x=
%$$

As a first application, we recover essentially \cite[Proposition 4]{MR3552249}:
\begin{cor}\label{cor:Brecurrentuniversal}
Let $(X,T)$ be a dynamical system which admits a universal $\mathsf{B}$-recurrent point. 
Then $T$ is topologically ergodic, that is, $N(U,V)$ is syndetic 
%\textup{(}i.e., has bounded gaps\textup{)} 
for all nonempty open sets $U,V \subseteq X$.
\end{cor}
\begin{proof} 
Thanks to Theorem \ref{thm:reiterativeequalhypercyclic}, $T$ is $\mathsf{B}$-universal. 
The conclusion follows by Proposition \ref{prop:differencesets} and the well-known result due to Erd{\H o}s and S\'{a}rk\"{o}zy that $S-S$ is syndetic whenever $S \in \mathsf{B}^+$, see e.g. \cite[Proposition 3.19]{MR0603625}.
\end{proof}

\begin{cor}\label{cor:weakmixingslkjf}
Let $(X,T)$ be a $\mathsf{B}$-universal dynamical system. 
%, where $X$ is Hausdorff. 
Then $N(U,V) \cap N(V,W) \neq \emptyset$ for all nonempty open sets $U,V,W \subseteq X$ such that $V$ contains a fixed point of $T$. 
\end{cor}
\begin{proof}
It follows by Corollary \ref{cor:mixingjshfkghj} and Corollary \ref{cor:Brecurrentuniversal}.
\end{proof}

In the context of linear dynamics, Corollary \ref{cor:weakmixingslkjf} simplifies and extends \cite[Theorem 2.2]{MR2153969} (although the underlying ideas are the same):
\begin{cor}
Let $(X,T)$ be a linear dynamical system, where $X$ is a Fr\'{e}chet space. If $T$ is $\mathsf{B}$-universal then $T\oplus T$ is universal. 
\end{cor}
%If a bounded linear operator $T$ on a Fr\'{e}chet space $X$ is $\mathsf{B}$-universal then $T\oplus T$ is universal. 
\begin{proof}
Thanks to Corollary \ref{cor:weakmixingslkjf}, $N(U,V) \cap N(V,W) \neq \emptyset$ for all nonempty open sets $U,V,W \subseteq X$ such that $0 \in V$. According to \cite[Theorem 2.47]{MR2919812}, the latter condition is equivalent to the universality of $T\oplus T$.
\end{proof}

We show now that, if $x$ is not periodic, then the gaps of $N(x,U)$ can be bounded from below.
%\textcolor{red}{Generalization of \cite[Proposition 3]{MR3552249}, vedi nuovo di Grivaux commentato+ref}
%https://arxiv.org/pdf/2212.03652.pdf
\begin{prop}\label{prop:smallestgap}
Let $(X,T)$ be a dynamical system and fix $x \in X$ which is not periodic. 
%Then there exists an open neighborhood $U$ of $x$ such that $N_T(x,U)$ has upper Banach density $0$. 
%Then for all $\varepsilon>0$ there exists a neighborhood $U$ of $x$ such that $N_T(x,U)$ has upper Banach density smaller than $\varepsilon$. 
Then for all $k\in \omega$ there exists 
%$h \in \omega$ and 
an open neighborhood $U$ of $x$ such that $|s-t|>k$ for all distinct $s,t \in N(x,U)$. 
%$N(x,U) \subseteq k\cdot \omega +h$. 
%, if $s,t \in N_T(x,U)$ with $s>t$, then $s>t+k$. 
\end{prop}
\begin{proof}
Fix a nonzero $k\in \omega$. Since $x$ is not periodic then $T^nx\neq x$ for all $n\ge 1$.
By the continuity of $T$ 
%for each $k\ge 1$ 
there exists an open neighborhood $U$ of $x$ such that $T^n[U]\cap U=\emptyset$ for all $n \in \{1,\ldots,k\}$. 
%Without loss of generality, $U_{k+1}\subseteq U_k$ for all $k$. Let us suppose for the sake of contradiction that $\lim_k\mathrm{bd}^\star(N_T(x,U_k))>1/k_0$ for some $k_0 \ge 1$.  
Suppose for the sake of contradiction that there exist $s,t \in N(x,U)$ with $1\le s-t \le k$. Note that $T^{s-t}[U] \cap U=\emptyset$ by the standing hypothesis. 
However, we have $y:=T^tx \in U$ and $T^{s-t}y=T^sx \in U$, so that $y \in T^{s-t}[U] \cap U$. 
\end{proof}

%As an immediate corollary:%, we extend \cite[Proposition 2 and Proposition 3]{MR3552249}:
\begin{cor}\label{cor:banachintornox}
Let $(X,T)$ be a dynamical system and fix $x \in X$ which is not periodic. Then for all $\varepsilon>0$ there exists a  neighborhood $U$ of $x$ such that $\mathsf{bd}^\star(N(x,U))\le \varepsilon$. 
% has upper Banach density smaller than $\varepsilon$. 
\end{cor}
\begin{proof}
%Since $x$ is nonperiodic then $T^nx\neq x$ for all $n\ge 1$. 
Fix an integer $k\ge 1/\varepsilon$. It follows by Proposition \ref{prop:smallestgap} that there exists a neighborhood $U$ of $x$ such that 
$$
\mathrm{bd}^\star(N(x,U))\le \max_{n \in \omega}\frac{|N(x,U)\cap [n,n+k)|}{k}\le \frac{1}{k}\le \varepsilon,
$$
concluding the proof. 
\end{proof}

As another consequence, we obtain a generalization of \cite[Remark 4.8]{MR3255465}:
\begin{thm}\label{thm:nonexistencestronglyuniversal}
Let $(X,T)$ be a dynamical system. 
%Then there are no $\mathsf{B}$-strong recurrent nonperiodic points. 
Then each point is $\mathsf{B}$-strong recurrent if and only if it is periodic.
% Let also $\mathsf{I}$ be an ideal on $\omega$ which contains $\mathsf{Bd}$. Then $T$ is not $\mathsf{I}$-strong universal. In particular, $T$ is not $\mathsf{Z}$-strong universal. 
\end{thm}
\begin{proof}
Each periodic point is, of course, $\mathsf{B}$-strong recurrent. Conversely, let us suppose for the sake of contradiction that there exists a $\mathsf{B}$-strong recurrent nonperiodic point $x \in X$. Then there exists a set $S \in \mathsf{B}^+$ such that $\lim_{n\in S}T^nx=x$. Thanks to Corollary \ref{cor:banachintornox}, it is possible to fix a neighborhood $U$ of $x$ such that $0<\mathsf{bd}^\star(N(x,U))<\mathsf{bd}^\star(S)$. But this is a contradiction since $N(x,U)\supseteq S\setminus F$ for some finite set $F\in \mathrm{Fin}$ so that $\mathsf{bd}^\star(N(x,U)) \ge \mathsf{bd}^\star(S\setminus F)=\mathsf{bd}^\star(S)$. 
\end{proof}

\begin{cor}\label{thm:nonexistencebdstronguniversal}
Let $(X,T)$ be a dynamical system such that no finite subset of $X$ is dense \textup{(}i.e., $X$ has infinite density character\textup{)}. Then $T$ is not $\mathsf{B}$-strong universal. 
\end{cor}
\begin{proof}
Let us suppose for the sake of contradiction that $x \in X$ is $\mathsf{B}$-strong universal. 
Since finite subsets of $X$ are not dense, then $x$ is not periodic. In addition, $x$ is certainly $\mathsf{B}$-strong recurrent. The conclusion follows by Theorem \ref{thm:nonexistencestronglyuniversal}. 
\end{proof}

In particular, since $\mathsf{B}\subseteq \mathsf{Z}$, we get:
\begin{cor}\label{cor:noZstronguniversal}
Let $(X,T)$ be a dynamical where $X$ is a nonzero topological vector space. Then $T$ is not $\mathsf{Z}$-strong universal. 
\end{cor}

%\textcolor{red}{CONNECTIONS WITH RIGIDITY DEFINITIONS GRIVAUX e VELOCITY OF CONVERGENCE AT 0}.

%\textcolor{red}{COMMENTS}

\section{Further structural results and applications}
%Applications}
\label{sec:applications}

%\textcolor{red}{COMMENTS AND TITLE} 

\subsection{Dense supply of null orbits} We start by showing that, if $T$ is additive on a topological group $X$ and there exists a large set of points with null orbits, then for all $x \in X$ and all nonempty open sets $U,V\subseteq X$ there exists $y \in V$ such that the return set $N(y,U)$ contains $N(x,U_0)$ modulo finite sets, for a suitable $U_0\subseteq U$.

\begin{prop}\label{prop:endomorphism}
Let $(X,T)$ be a dynamical system, where $X$ is a topological group and $T$ is a homomorphism. Suppose also that there exists a dense $D\subseteq X$ such that $\lim_nT^nz=0$ for all $z \in D$. 

Then, for all nonempty open sets $U\subseteq X$ and for all points $u \in U$, there exists an open neighborhood $U_0\subseteq U$ of $u$ such that, for all points $x \in X$ and all nonempty open sets $V\subseteq X$, there exist $y \in V$ and a finite $F\in \mathrm{Fin}$ for which 
$$
%u \in U_0\subseteq U
%\quad \text{ and }\quad 
N(x,U_0)\setminus F\subseteq N(y,U).
$$
\end{prop}
%NON CANCELLARE: https://math.stackexchange.com/questions/1397059/symmetric-open-neighborhood-in-topological-group
\begin{proof}
Fix nonempty open sets $U,V \subseteq X$ and two points $x \in X$ and $u \in U$. 
%??? We need to show that there exists $y \in V$ ??? 
It is known by standard properties of topological groups that there exists a nonempty open neighborhood $W_0\subseteq X$ of $0$ with the property that 
$$
u+W_0+W_0\subseteq U. 
%\quad \text{ and }\quad 
%v+W_0+W_0\subseteq V,
$$
see e.g. \cite[Chapter 3]{MR0000265}. Notice that, since $D$ is dense, it is possible to pick a point $z \in D\cap (-x+V)$. To complete the proof, it is enough to show that $U_0:=u+W_0$ (so that $u \in U_0\subseteq U$) and $y:=x+z$ satisfy the required property for some finite $F$, i.e., $T^ny \in V$ for all sufficiently large integers $n$ such that $T^nx \in U_0$. For, it follows by construction that $y \in x+(-x+V)=V$ 
%$$
%y
%%=x+z 
%\in x+(-x+v+W_0)\subseteq v+W_0+W_0 \subseteq V
%$$
and, in addition, 
$$
T^ny=T^nx+T^nz \in U_0+W_0=(u+W_0)+W_0\subseteq U
$$
for all $n \in N(x,U_0)$ such that $T^nz \in W_0$. 
\end{proof}

%It is clear that the hypothesis $\lim_nT^nz=0$ for each $z \in D$ implies that, if $X$ is Hausdorff, then all nonzero points $z \in D$ cannot be periodic. This conclusion can be strenghtened also outside $D$ (in particular, these dynamical systems are not chaotic modulo the trivial case $X=\{0\}$): 
%\begin{cor}
%Let $(X,T)$ be a dynamical system, where $X$ is a Hausdorff topological group and $T$ is an endomorphism. Suppose also that there exists a dense $D\subseteq X$ such that $\lim_nT^nz=0$ for all $z \in D$. Then there are no nonzero periodic points. %In particular, $T$ is not chaotic unless $X=\{0\}$. 
%\end{cor}
%\begin{proof}
%Since $T$ is an endomorphism then $T(0)=0$. Now, let us suppose for the sake of contradiction that there exists a nonzero periodic point $y \in X$ and pick two disjoint open neighborhoods $U,V\subseteq X$ of $0$ and $y$, respectively. Thanks to Proposition \ref{prop:endomorphism}
%\end{proof}

The following corollary is immediate.
\begin{cor}
Let $(X,T)$ and $D$ as in Proposition \ref{prop:endomorphism}, and let $\mathsf{I}$ be an ideal on $\omega$. Let also $U\subseteq X$ be an open set containing an $\mathsf{I}$-recurrent point. 
Then, for each nonempty open set $V\subseteq X$, there exists $y \in V$ such that $N(y,U) \in \mathsf{I}^+$. 
\end{cor}
\begin{proof}
Pick a point $z \in U \cap \mathrm{Rec}_T(\mathsf{I})$ and set $x=u=z$ in Proposition \ref{prop:endomorphism}.
\end{proof}

Another simple corollary extends the known result that chaotic weighted shifts on Fr\'{e}chet sequence spaces $X$ with an unconditional basis are necessarily frequently hypercyclic (that is, there exists $x \in X$ such that $N(x,U)$ has positive lower asymptotic density for each nonempty open set $U\subseteq X$):
\begin{cor}\label{cor:chaoticimpliesfrequently}
Let $(X,T)$ and $D$ as in Proposition \ref{prop:endomorphism}, and suppose that there is a dense set $P\subseteq X$ of periodic points. 

Then, for each nonempty open set $U\subseteq X$, there exists a nonzero $k \in \omega$ such that, for all nonempty open sets $V\subseteq X$, there exist $y \in V$ and $F \in \mathrm{Fin}$ for which 
$$
(k\cdot \omega) \setminus F \subseteq N(y,U).
$$
\end{cor}
\begin{proof}
Let $U\subseteq X$ be a nonempty open set and fix a periodic point $z \in P\cap U$. Denoting with $k$ its period and setting $x=u=z$ in Proposition \ref{prop:endomorphism}, there exists a neighborhood $U_0$ of $z$ such that, for each nonempty open $V\subseteq X$, there exists $y \in V$ for which $N(y,U)$ contains $N(z,U_0)$, up to some finite set. The claim follows since $k\cdot \omega \subseteq 
%\{n \in \omega: T^nz=z\}\subseteq 
N(z,U_0)$. 
\end{proof}

As another application, we obtain an extension of \cite[Theorem 8.5]{MR4489276} and \cite[Corollary 5.20]{MR4238631}, providing a sufficient condition for the equivalence between $\mathsf{I}$-recurrence and $\mathsf{I}$-universality. This result has to be confronted with Theorem \ref{thm:upperfrequentcharacterization}, where one of the equivalent conditions is that the set of $\mathsf{I}$-recurrent points is not meager. 
\begin{thm}\label{thm:IrecurrentIhypercyclic}
Let $(X,T)$ be a dynamical system, where $X$ is a Baire second countable group and $T$ is a homomorphism. Let $\mathsf{I}$ be a strongly-right-translation invariant analytic $P$-ideal and suppose that there exists a dense $D\subseteq X$ such that $\lim_nT^nz=0$ for all $z \in D$. 
Then $T$ is $\mathsf{I}$-hypercyclic if and only if it is $\mathsf{I}$-recurrent. 
%Then the following are equivalent\textup{:}
%\begin{enumerate}[label={\rm (\roman{*})}]
%\item \label{item:1recuniv} $\mathrm{Univ}_T(\mathsf{I})$ is comeager\textup{;}
%\item \label{item:2recuniv} $T$ is $\mathsf{I}$-universal\textup{;}
%\item \label{item:3recuniv} $T$ is $\mathsf{I}$-recurrent\textup{;}
%\item \label{item:4recuniv} for each nonempty $U\subseteq X$, there exists $\kappa>0$ such that for all nonempty open $V\subseteq X$ there exists $y \in V$ for which $\|N(y,U)\|_\varphi \ge \kappa$.%, where $\mathsf{I}=\mathrm{Exh}(\varphi)$ as in Definition \ref{defi:stronginvariant}. 
%\end{enumerate}
\end{thm} 
\begin{proof}
Suppose that $T$ is $\mathsf{I}$-recurrent and let $\varphi$ be a lscsm such that $\mathsf{I}=\mathrm{Exh}(\varphi)$ as in Definition \ref{defi:stronginvariant}. Pick a nonempty open set $U\subseteq X$. Hence we can find an $\mathsf{I}$-recurrent point $z \in U$. Now, set $x=u=z$ in Proposition \ref{prop:endomorphism}. It follows that there exists a neighborhood $U_0\subseteq U$ of $z$ such that, for all nonempty open sets $V \subseteq X$ there exist $y \in V$ and $F\in \mathrm{Fin}$ for which $N(z,U_0)\setminus F\subseteq N(y,U)$. Define $\kappa:=\|N(z,U_0)\|_\varphi>0$. Recalling that $\|\cdot\|_\varphi$ is monotone and invariant modulo finite sets, for all nonempty open sets $V \subseteq X$ there exists $y \in V$ such that $\|N(y,U)\|_\varphi \ge \kappa$. 
%, where $\kappa:=\|N(z,U_0)\|_\varphi>0$. 
Thanks to Proposition \ref{prop:birkoff}, we conclude that $T$ is $\mathsf{I}$-universal. The converse is obvious. 
\end{proof}

An additional equivalent condition will be obtained later in Theorem \ref{thm:zerooneorbitlkfj}.

Even if Theorem \ref{thm:IrecurrentIhypercyclic} does not contain the case $\mathsf{I}=\mathsf{B}$ (since $\mathsf{B}$ is not a $P$-ideal), we can still recover it with the aid of previous results:
\begin{cor}\label{cor:BrecurrenceBuniversality}
Let $(X,T)$ and $D$ as in Theorem \ref{thm:IrecurrentIhypercyclic}. Then $T$ is $\mathsf{B}$-universal if and only if it is $\mathsf{B}$-recurrent.
\end{cor}
\begin{proof}
Suppose that $T$ is $\mathsf{B}$-recurrent. Then it is, of course, recurrent and, thanks to Theorem \ref{thm:IrecurrentIhypercyclic}, it has to be universal (of course, $\mathrm{Fin}$ is a strongly-right-translation invariant analytic $P$-ideal). It follows by the equivalence \ref{item:2Buniv} $\Longleftrightarrow$ \ref{item:4Buniv} in Theorem \ref{thm:reiterativeequalhypercyclic} that $T$ is $\mathsf{B}$-universal. The converse is obvious. 
\end{proof}

For the next results, we will need the following characterization of $\mathsf{I}$-limit points:
\begin{thm}\label{thm:mareklimit}
Let $\bm{x}$ be a sequence taking values in a first countable  space $X$. Also, let $\mathsf{I}=\mathrm{Exh}(\varphi)$ be an analytic P-ideal on $\omega$ as in \eqref{eq:characterizationanalPideal} and, for each $\eta \in X$, denote by $(U_{\eta,k}: k \in \omega)$ a decreasing local base at $\eta$. 

Define $\mathfrak{u}(\eta):= \lim\nolimits_{k}\|\{n \in \omega: x_n \in U_{\eta,k}\}\|_\varphi$ for each $\eta \in X$.  
Then 
$$
\Lambda_{\bm{x}}(\mathsf{I})=\left\{\eta\in X:  
\mathfrak{u}(\eta) 
%\lim_{k \to \infty}\|\{n \in \omega: x_n \in U_{\eta,k}\}\|_\varphi 
>0\right\}
%\quad \text{ where }\quad \mathfrak{u}(\eta):=\lim_{k\to \infty}\|\{n \in \omega: x_n \in U_{\eta,k}\}\|_\varphi
$$ 
and the map 
$\mathfrak{u}$ 
%the map $\mathfrak{u}: X\to [0,\infty)$ defined by $\eta \mapsto \lim\nolimits_{k}\|\{n \in \omega: x_n \in U_{\eta,k}\}\|_\varphi$ 
is lower semicontinuous. 

In addition, if $\eta \in \Lambda_{\bm{x}}(\mathsf{I})$, then $\lim_{n \in A}x_n=\eta$ for some $A\subseteq \omega$ such that $\|A\|_\varphi \ge \mathfrak{u}(\eta)$. 
%In addition, $\mathfrak{u}$ is upper semi-continuous. 
\end{thm}
\begin{proof}
See 
\cite[Theorem 2.2]{MR3883171} and its proof. 
%\cite[Lemma 2.1 and Theorem 2.2]{MR3883171}. %%\textcolor{red}{SERVE A QUESTO PUNTO? ESEMPI DI NON P+ IDEALI?}
\end{proof}

We are now able to show that the set of points with a prescribed $\mathsf{I}$-limit point is either empty or comeager.

\begin{thm}\label{thm:denseIimitpoints}
Let $(X,T)$ be a dynamical system, where $X$ is a completely metrizable group and $T$ is a homomorphism. Let $\mathsf{I}$ be an analytic $P$-ideal and suppose that there exists a dense $D\subseteq X$ such that $\lim_nT^nz=0$ for all $z \in D$. 
Then, for all $x_0 \in X$, %the set 
\begin{equation}\label{eq:claimedsetnicesetsd}
\{y \in X: x_0 \text{ is an }\mathsf{I}\text{-limit point of }\mathrm{orb}(y,T)\}
\end{equation}
is either empty or comeager. 
%whenever there exists $x \in X$ such that 
%$x_0$ is an $\mathsf{I}$-limit point of $\mathrm{orb}(x,T)$. 
%$x_0 \in\Lambda_{\mathrm{orb}(x,T)}(\mathsf{I})$. 
%for each $\mathsf{I}$-strong recurrent point $x_0 \in X$.
\end{thm}

\begin{proof}
%[Proof of Theorem \ref{thm:denseIimitpoints}]
Let $\varphi$ be a lscsm such that $\mathsf{I}=\mathrm{Exh}(\varphi)$ as in \eqref{eq:characterizationanalPideal}, let $d$ be a compatible metric on $X$. 
Also, fix $x,x_0$ such that $x_0$ is an $\mathsf{I}$-limit point of $\mathrm{orb}(x,T)$ and let $(U_k)$ be a decreasing local base at $x_0$. Thanks to Theorem \ref{thm:mareklimit}, this means that 
$$
\kappa:=
%\mathfrak{u}(x_0)=
\lim_{k\to \infty} \|N(x,U_k)\|_\varphi>0.
$$

We are going to use the Banach--Mazur game defined as follows: 
Players I and II choose alternatively nonempty open subsets of $X$ as a decreasing chain 
$$
V_0\supseteq W_0 \supseteq V_1 \supseteq W_1\supseteq \cdots, 
$$
where Player I chooses the sets $V_0,V_1,\ldots$; Player II is declared to be the winner of the game if 
\begin{equation}\label{eq:claimgame}
\bigcap\nolimits_{m\in \omega} W_m \cap \, Y\neq \emptyset,
\end{equation}
where $Y$ stands for the set in \eqref{eq:claimedsetnicesetsd}. 
Then Player II has a winning strategy (that is, he is always able to choose suitable sets $W_0, W_1,\ldots$ so that \eqref{eq:claimgame} holds at the end of the game) if and only if $Y$ is a comeager subset of $X$, see \cite[Theorem 8.33]{MR1321597}. 

%We need to show that, for each nonempty open set $V\subseteq X$ there exists $y \in V$ such that $x_0$ is an $\mathsf{I}$-limit point of the sequence $\mathrm{orb}(y,T)$. 
For, let us describe the winning strategy of Player II. It follows by Proposition \ref{prop:endomorphism} that there exist $y_0 \in V_0$ and $j_0 \in \omega$ such that $N(y_0,U_0)$ contains, up to some finite set, $N(x,U_{j_0})$.
Since $\varphi$ is a lscsm, we obtain
\begin{equation}\label{eq:lscsmconsequence1}
\lim_{n\to \infty}\varphi(N(y_0,U_0)\cap [0,n])\ge \|N(y_0,U_0)\|_\varphi \ge \|N(x,U_{j_0})\|_\varphi \ge \kappa.
\end{equation}
Hence there exists a finite set $F_0\subseteq N(y_0,U_0)$ such that $\varphi(F_0) \ge \kappa/2$. Since $T$ is continuous, there exists an open ball $B(y_0,r_{0})$ with radius $r_{0}\le 2$ such that 
$B(y_0,r_{0})\subseteq V_0$ and $T^n[B(y_0,r_{0})] \subseteq U_{0}$ 
for all $n \in F_{0}$. Hence, define
$$
W_{0}:=B(y_0,r_{0}/2).
$$
It follows that $W_{0}\subseteq \tilde{W}_{0}\subseteq V_0$ and $T^n[\tilde{W}_{0}] \subseteq U_{0}$ 
%$$
%W_{0}\subseteq \tilde{W}_{0}\subseteq V_0
%\quad \text{ and }\quad 
%T^n[\tilde{W}_{0}] \subseteq U_{0}
%$$
for all integers $n \in F_{0}$, where $\tilde{W}_{0}$ stands for the closed ball $\{x \in X: d(x,y_0) \le r_{0}/2\}$.

Next, suppose that the open sets $V_0\supseteq W_0 \supseteq V_1 \supseteq W_1\supseteq \cdots \supseteq V_m$ have been chosen, for some integer $m\ge 1$, and proceed similarly as before:
\begin{enumerate}[label={\rm (\roman{*})}]

\item Thanks to Proposition \ref{prop:endomorphism}, fix $y_{m} \in V_{m}$ and an integer $j_{m}>j_{m-1}$ such that $N(y_{m}, U_{m})$ contains, up to some finite set, $N(x, U_{j_{m}})$. 

\item Analogously to \eqref{eq:lscsmconsequence1}, we have
\begin{displaymath}
\begin{split}
\lim_{n\to \infty}\varphi(N(y_{m},U_{m})\cap (\max F_{m-1} ,n])&\ge \|N(y_{m},U_{m})\|_\varphi\\
& \ge \|N(x,U_{j_{m}})\|_\varphi \ge \kappa. 
\end{split}
\end{displaymath}
Hence it is possible to fix a finite set $F_{m} \subseteq N(y_{m},U_{m})$ such that 
$$
\min F_{m}>\max F_{m-1}
\quad \text{ and }\quad 
\varphi(F_{m}) \ge \kappa (1-2^{-m}).
$$

\item By the continuity of $T$, there exists a radius $r_m \le 2^{1-m}$ such that the open ball $B(y_m,r_{m})$ is contained in $V_m$ and, in addition, $T^n[B(y_m,r_{m})] \subseteq U_{m}$ 
for all $n \in F_{m}$. 
Hence, define
$$
W_{m}:=B(y_m,r_{m}/2).
$$
Note that the radius of $W_{m}$ is $\le 2^{-m}$ and that $W_{m}$ is contained in the closed ball $\tilde{W}_{m}:=\{x \in X: d(x,y_m) \le r_{m}/2\}$. This implies that 
\begin{equation}\label{eq:inclusionchainrecursion}
W_{m}\subseteq \tilde{W}_{m}\subseteq V_m 
\quad \text{ and }\quad 
T^n[\tilde{W}_{m}] \subseteq U_{m}
\end{equation}
for all $n \in F_{m}$.
\end{enumerate}

At this point, since the sequence $(W_m)$ is decreasing and $y_m \in W_m$ for all $m \in \omega$, we obtain  $d(y_i,y_j) \le 2^{-\min\{i,j\}}$ for all $i,j \in \omega$, hence the sequence $(y_m)$ is Cauchy. The completeness of $X$ (and of all $\tilde{W}_m$) implies the existence of 
$$
y:=\lim_{m\to \infty} y_m \in X,
$$ %Notice that, by construction, $d(y,y_m) \le 2^{-m}$ for all $m \in \omega$. 
so that $\{y\}=\bigcap\nolimits_m \tilde{W}_{m}=\bigcap\nolimits_m W_{m}$. In particular, we obtain by \eqref{eq:inclusionchainrecursion} that 
$
T^n y \in U_{m} \subseteq U_k
$ 
for all $m,k \in \omega$ with $m\ge k$ and for all $n \in F_m$. This implies that 
$$
\bigcup\nolimits_{m\ge k}F_m \subseteq N(y,U_k)
$$
for all $k \in \omega$. Since $\varphi$ is a lscsm, we get
\begin{displaymath}
\lim_{k\to \infty}\|N(y,U_k))\|_\varphi \ge \left\|\bigcup\nolimits_{m\ge k}F_m\right \|_\varphi
\ge \limsup_{m\to \infty}\varphi(F_m) \ge \kappa
\end{displaymath}
for all $k \in \omega$, hence $\lim_{k\to \infty}\|N(y,U_k))\|_\varphi>0$. Thanks to Theorem \ref{thm:mareklimit}, $x_0$ is an $\mathsf{I}$-limit point of $\mathrm{orb}(y,T)$, therefore $y \in Y$, which completes the proof. 
\end{proof}

The following corollary is clear from Theorem \ref{thm:denseIimitpoints}:
\begin{cor}\label{cor:emptyordenselinear}
Let $(X,T)$ be a linear dynamical system, where $X$ is a Fr\'{e}chet space. 
%and $T$ is a homomorphism. 
Suppose that there exists a dense $D\subseteq X$ such that $\lim_nT^nz=0$ for all $z \in D$. Then, for all $x_0 \in X$, the set 
$$
\{y \in X: \lim\nolimits_{n \in A}T^ny=x_0 \text{ for some }A \in \mathsf{Z}^+\}
$$
is either empty or comeager. 
\end{cor}

\begin{rmk}\label{rmk:strongerclaimdenselimitpoints}
The proof of Theorem \ref{thm:denseIimitpoints} shows something stronger: fix a point $x_0 \in X$, and suppose that the set in \eqref{eq:claimedsetnicesetsd} is nonempty. Then there exists a comeager set of points $y \in X$ such that $\lim_k \|N(y,U_k)\|_\varphi \ge \lim_k \|N(x,U_k)\|_\varphi$, where $(U_k)$ is a decreasing local base at $x_0$. 
\end{rmk}

\begin{rmk}\label{rmk:x_0strongperiodic}
Note that, in the statement of Theorem \ref{thm:denseIimitpoints}, $X$ is not assumed to be separable. In particular, it applies to all $\mathsf{I}$-strong recurrent points $x_0$ (hence, also to periodic points, which are $\mathsf{I}$-strong recurrent provided e.g. that $\mathsf{I}$ is left- or right-translation invariant). %Accordingly, thanks to Theorem \ref{thm:nonexistencestronglyuniversal}, recall that if $\mathsf{B}\subseteq \mathsf{I}$ then every $\mathsf{I}$-strong recurrent point is certainly periodic.
\end{rmk}

The technique used in the proof of Theorem \ref{thm:denseIimitpoints} does not seem to adapt to the $\mathsf{I}$-cluster points analogue. The latter, however, holds if $\mathsf{I}$ is a $F_\sigma$ $P$-ideal:
\begin{cor}\label{cor:denseclusterpoints}
Let $(X,T)$, and $D$ as in Theorem \ref{thm:denseIimitpoints}. Let also $\mathsf{I}$ be a $F_\sigma$ $P$-ideal on $\omega$. %
Then, for all $x_0 \in X$, 
\begin{equation}\label{eq:claimedsetnicesetsd2}
\{y \in X: x_0 \text{ is an }\mathsf{I}\text{-cluster point of }\mathrm{orb}(y,T)\}
\end{equation}
is either empty or comeager. 
%whenever there exists $x \in X$ such that 
%$x_0$ is an $\mathsf{I}$-cluster point of $\mathrm{orb}(x,T)$. 
\end{cor}
\begin{proof}
This follows by Theorem \ref{thm:denseIimitpoints} and the fact that $\mathsf{I}$-cluster points coincide with $\mathsf{I}$-limit points if $X$ is first countable and $\mathsf{I}$ is a $F_\sigma$-ideal, see \cite[Theorem 2.3]{MR3883171}. 
\end{proof}

%\textcolor{red}{COMMENTS}

It is worth noting that the denseness of the sets \eqref{eq:claimedsetnicesetsd} and \eqref{eq:claimedsetnicesetsd2} hold under much weaker hypotheses, regardless of the complexity of $\mathsf{I}$ and the completeness of $X$:
\begin{prop}\label{prop:denseIimitpoints}
Let $(X,T)$ be a dynamical system, where $X$ is a topological group and $T$ is a homomorphism. Let $\mathsf{I}$ be an ideal on $\omega$ and suppose also that there exists a dense $D\subseteq X$ such that $\lim_nT^nz=0$ for all $z \in D$. 
Then, for all $x_0 \in X$, each of the sets
$$
\{y \in X: x_0 \text{ is an }\mathsf{I}\text{-limit point of }\mathrm{orb}(y,T)\}
$$
and 
$$
\{y \in X: x_0 \text{ is an }\mathsf{I}\text{-cluster point of }\mathrm{orb}(y,T)\}
$$
is either empty or dense.  
\end{prop}
\begin{proof}
Fix $x_0 \in X$ and suppose that $x_0$ is an $\mathsf{I}$-limit point of $\mathrm{orb}(x,T)$ for some $x \in X$, so that there exists $A \in \mathsf{I}^+$ such that $\lim_{n \in A}T^nx=x_0$. Let $U\subseteq X$ be a nonempty open set. 
Fix a point $z \in (-x+U) \cap D$ and define $y:=x+z$. For each open neighborhood $V$ of $x_0$, there exists an open neighborhood $W_0\subseteq X$ of $0$ with the property that $x_0+W_0+W_0\subseteq V$, see e.g. \cite[Chapter 3]{MR0000265}. It follows that $T^n y=T^n x+T^n z \in (x_0+W_0)+W_0$ for all sufficiently large integers in $A$. Therefore $x_0$ is an $\mathsf{I}$-limit point of $\mathrm{orb}(y,T)$. 

The proof for the second part goes analogously (we omit details). 
\end{proof}

%\textcolor{red}{APPLICATION PREVIOUS RESULTS; ZERO ONE LAWS FOR ANALYTIC P IDEALS} 
Now we show that, with similar hypotheses, $\mathsf{I}$-universality and $\mathsf{I}$-recurrence coincide with the property that each nonempty open set contains a $\mathsf{I}$-cluster point of some orbits, hence extending our previous Theorem \ref{thm:IrecurrentIhypercyclic}.

\begin{thm}\label{thm:zerooneorbitlkfj}
Let $(X,T)$ be a dynamical system, where $X$ is a Baire second countable group and $T$ is a homomorphism. Also, let $\mathsf{I}$ be a strongly-right-translation invariant analytic $P$-ideal and suppose that there exists a dense $D\subseteq X$ such that $\lim_n T^n z=0$ for all $z \in D$. 

Then the following are equivalent\textup{:}
\begin{enumerate}[label={\rm (\roman{*})}]
\item \label{item:2orbit} $T$ is $\mathsf{I}$-universal\textup{;}
\item \label{item:4orbit} $T$ is $\mathsf{I}$-recurrent\textup{;}
\item \label{item:5orbit} For each nonempty open set $U\subseteq X$, there exists $x \in X$ such that its orbit $\mathrm{orb}(x,T)$ has an $\mathsf{I}$-cluster point in $U$\textup{.}
\end{enumerate}
If, in addition, $\mathsf{I}$ is a $F_\sigma$-ideal, they are also equivalent to\textup{:}
\begin{enumerate}[label={\rm (\roman{*})}]
\setcounter{enumi}{3}
\item \label{item:1orbit} $T$ is $\mathsf{I}$-strong universal\textup{;}
\item \label{item:3orbit} $T$ is $\mathsf{I}$-strong recurrent\textup{.}
\end{enumerate}
\end{thm}
\begin{proof}
Theorem \ref{thm:IrecurrentIhypercyclic} proves the equivalence \ref{item:2orbit} $\Longleftrightarrow$ \ref{item:4orbit}. It is obvious that \ref{item:2orbit} $\implies$ \ref{item:5orbit}. Conversely, let us show \ref{item:5orbit} $\implies$ \ref{item:2orbit}. 
To this aim, pick a nonempty open set $U\subseteq X$ and fix $x\in X$ and $x_0 \in U$ such that $x_0$ is an $\mathsf{I}$-cluster point of $\mathrm{orb}(x,T)$. 
Define $\kappa:=\|N(x,U)\|_\varphi>0$. At this point, by Proposition \ref{prop:denseIimitpoints} and its proof, for each nonempty open $V\subseteq X$ there exists $y \in V$ such that $x_0$ is an $\mathsf{I}$-cluster point of $\mathrm{orb}(y,T)$ (in particular, $U \cap \Gamma_{\mathrm{orb}(y,T)}(\mathsf{I})\neq \emptyset$) and, more precisely, $\|N(y,U)\|_\varphi\ge \kappa$ because $N(y,U)$ contains a cofinite subset of $N(x,U)$. 
%
%Thanks to Theorem \ref{thm:mareklimit}, we have 
%$$
%\kappa:=\lim_{k\to \infty} \|N(x,U_k)\|_\varphi>0,
%$$
%where $(U_k)$ is a decreasing local base at $x_0$. At this point, for each nonempty open $V\subseteq Y$ there exists $y \in V$ such that $x_0$ is an $\mathsf{I}$-limit point of $\mathrm{orb}(y,T)$ (in particular, $U \cap \Gamma_{\mathrm{orb}(y,T)}(\mathsf{I})\neq \emptyset$) and, more precisely, $\lim_k \|N(y,U_k)\|_\varphi\ge \kappa$, see Proposition \ref{prop:denseIimitpoints} and its proof. 
We conclude by Proposition \ref{prop:birkoff} that $T$ is $\mathsf{I}$-universal. 

Lastly, 
%in case $\mathsf{I}$ is a $F_\sigma$-ideal, 
the equivalences \ref{item:2orbit} $\Longleftrightarrow$ \ref{item:1orbit} and \ref{item:4orbit} $\Longleftrightarrow$ \ref{item:3orbit} follow by \cite[Theorem 2.3]{MR3883171}.
\end{proof}

%\textcolor{red}{Reference to zero-one law orbits. Nel successivo non assunto nonzero}

%\textcolor{red}{APPLICATION PREVIOUS RESULTS; ZERO ONE LAWSFOR F SIGMA IDEALS}

In the next result, we prove that the technical condition of being strongly-right-translation invariant can be removed in case of $F_\sigma$ $P$-ideals. 
\begin{thm}\label{thm:zerooneorbitlkfjFsigma}
Let $(X,T)$ be a dynamical system, where $X$ is a Polish group and $T$ is a homomorphism. Also, let $\mathsf{I}$ be a $F_\sigma$ $P$-ideal and suppose that there exists a dense $D\subseteq X$ such that $\lim_n T^n z=0$ for all $z \in D$. 

Then the following are equivalent\textup{:}
\begin{enumerate}[label={\rm (\roman{*})}]
\item \label{item:A1orbit} $T$ is $\mathsf{I}$-strong universal\textup{;}
\item \label{item:A2orbit} $T$ is $\mathsf{I}$-universal\textup{;}
\item \label{item:A3orbit} For each nonempty open set $U\subseteq X$, there exists $x \in X$ such that its orbit $\mathrm{orb}(x,T)$ has an $\mathsf{I}$-cluster point in $U$.
\end{enumerate}
%If, in addition, $\mathsf{I}$ is strongly-right-translation invariant, they are also equivalent to\textup{:}
%\begin{enumerate}[label={\rm (\roman{*})}]
%\setcounter{enumi}{3}
%\item \label{item:A4orbit} $T$ is $\mathsf{I}$-strong recurrent\textup{;}
%\item \label{item:A5orbit} $T$ is $\mathsf{I}$-recurrent\textup{.}
%\end{enumerate}
\end{thm}
\begin{proof}
The equivalence \ref{item:A1orbit} $\Longleftrightarrow$ \ref{item:A2orbit} follows by \cite[Theorem 2.3]{MR3883171}. 
%The equivalences \ref{item:A1orbit} $\Longleftrightarrow$ \ref{item:A2orbit} and \ref{item:A4orbit} $\Longleftrightarrow$ \ref{item:A5orbit} follow by \cite[Theorem 2.3]{MR3883171}. 
It is obvious that \ref{item:A2orbit} $\implies$ \ref{item:A3orbit}. Conversely, let us show \ref{item:A3orbit} $\implies$ \ref{item:A2orbit}. For, pick a countable dense subset $\{\eta_k: k \in \omega\} \subseteq X$. 
Thanks to Corollary \ref{cor:denseclusterpoints}, for each $k \in \omega$, there exists a comeager set $G_k\subseteq X$ such that $\eta_k$ is an $\mathsf{I}$-cluster point of $\mathrm{orb}(y,T)$ for all $y \in G_k$. At this point, define the comeager set 
$
G:=\bigcap\nolimits_{k}G_k.
$ 
%$
%G:=\bigcap\nolimits_{k \in \omega}G_k.
%$ 
For each $y \in G$, it follows that $\Gamma_{\mathrm{orb}(y,T)}(\mathsf{I})$ contains $\{\eta_n: n \in \omega\}$ and, in addition, it is closed by \cite[Lemma 3.1(iv)]{MR3920799}. Therefore $\Gamma_{\mathrm{orb}(y,T)}(\mathsf{I})=X$ whenever $y \in G$, so that $G\subseteq \mathrm{Univ}_{T}(\mathsf{I})$. 
%
%Lastly, the equivalence \ref{item:A2orbit} $\Longleftrightarrow$ \ref{item:A5orbit} follows by Theorem \ref{thm:IrecurrentIhypercyclic}. 
\end{proof}

Even if it is not needed in 
%the proofs of 
Corollary \ref{cor:denseclusterpoints} and Theorem \ref{thm:zerooneorbitlkfjFsigma}, it is possible to characterize $F_\sigma$ $P$-ideals on the same lines of \eqref{eq:mazurcharacter} and \eqref{eq:characterizationanalPideal}, see \cite[Theorem 1.2.5(c)]{MR1711328}. 

Also, it is worth noting that a weaker property than the above condition \ref{item:A3orbit} has been considered by Costakis and Parissis in \cite{MR2943720}. The latter only requires that, for each nonempty open set $U\subseteq X$ there exists $x \in X$ such that $N(x,U) \in \mathsf{I}^+$, cf. also \cite{MR3742548} and \cite[Proposition 4.6]{MR2323544}. 

Moreover, the last two results seem to be somehow related the phenomenon called \textquotedblleft zero-one law of orbital limit points\textquotedblright\, initiated by Chan and Seceleanu \cite{MR2881542}, where they proved that, if $T$ is a (unilateral) weighted shift on $\ell_p$ (with $1\le p<\infty$), then $T$ is hypercyclic if and only if there exists an orbit with a nonzero limit point; cf. also \cite{BonillaErdmann} and references therein. We remark that the latter characterization is rather special and does not hold for general operators, 
%even for bilateral backward weighted shifts, 
see \cite[Section 7]{Abak} and \cite[Proposition 2.2]{MR3490772}. 

%\textcolor{red}{COMMENTI SUL PARAMETRO}

\begin{defi}
Given a first countable dynamical system $(X,T)$, a decreasing local base $(U_k)$ at $\eta \in X$, and a lscsm $\varphi: \mathcal{P}(\omega) \to [0,\infty)$, define the parameter
$$
c_\varphi(T, \eta):=\inf\nolimits_{k} \sup \left\{ \| N(x,U_k)\|_\varphi: x \text{ is a universal point}\right\}.
$$
\end{defi}

In the special case where $X$ is a Banach space, $\eta=0$, and $\mathrm{Exh}(\varphi)=\mathsf{Z}$, a similar definition (which, however, relies on the explicit definition of $\mathsf{d}^\star$) appeared in \cite{MR3255465}, cf. also \cite{MR4238631}. 
Therein, the study of this \textquotedblleft natural parameter\textquotedblright\, has been useful 
%in \cite{MR3255465} 
to show the existence of a frequently hypercyclic operator on $c_0(\mathbb{Z})$ which
does not admit any ergodic measure with full support. %In the following, $\partial U$ denotes the boundary of $U$. 

\begin{prop}\label{prop:lscsminclusion}
Let $(X,T)$ be a dynamical system, where $X$ is a first countable space, and fix a nonempty open set $U\subseteq X$. Let also $x \in X$ be a universal point 
% for which $T^nx\notin \partial U$ for all $n \in\omega$, 
and fix a lscsm $\varphi: \mathcal{P}(\omega) \to [0,\infty)$ such that $\|\cdot\|_\varphi$ is translation invariant. 
%
%Then 
%$$
%\left\{y \in X: \|N(y,U)\|_\varphi \ge \|N(x,U)\|_\varphi \right\}
%$$
%contains a dense $G_\delta$-set.
%
Then 
$$
\left\{y \in X: \|N(y,U)\|_\varphi \ge \|N(x,U)\|_\varphi \right\}
$$
contains a dense $G_\delta$-set.
\end{prop}
\begin{proof}
The claimed set contains $S:=\bigcap_{i,j}\bigcup_{k} S_{i,j,k}$, where 
%$S_{i,j,k}$ is defined by 
\begin{displaymath}
\begin{split}
S_{i,j,k}:=
\left\{y \in X: \varphi(N(y,U) \cap [j,j+k]\,) \ge 
(1-2^{-i})\cdot \|N(x,U)\|_\varphi
%\frac{\|N(y,U)\|_\varphi}{i+1} 
\right\}
\end{split}
\end{displaymath}
%\begin{displaymath}
%\begin{split}
%S_{i,j,k}:=
%\left\{y \in X: \varphi(N(y,U) \cap [j,j+\right. & k]\,) \ge 
%(1-2^{-i})\cdot \|N(x,U)\|_\varphi
% \\
% &\left. \text{ and } T^ny\notin \partial U \text{ for all }n\in \{1,\ldots, j+k\}\right\}
%\end{split}
%\end{displaymath}
for all $i,j,k \in \omega$. By the continuity of $T$, 
% and the conditions $T^ny\notin \partial U$, 
it is easy to see that each point in $S_{i,j,k}$ is an interior point, hence $S$ is $G_{\delta}$. Moreover, since $\|\cdot\|_\varphi$ is translation invariant, the claimed set contains $T^nx$ for all $n \in \omega$. Therefore it is dense. 
\end{proof}

The following result provides the nonlinear version of \cite[Proposition 4.7]{MR3255465}. 
\begin{thm}\label{cor:Ilimitswithcvarphi}
Let $(X,T)$ be a universal dynamical system, where $X$ is a Baire round metric space, and fix a lscsm $\varphi: \mathcal{P}(\omega) \to [0,\infty)$ with $\|\cdot\|_\varphi$ translation invariant. Then 
$$
\left\{y \in 
X: 
%\mathrm{Univ}_T(\mathrm{Fin}): 
\lim_{n \in A}T^ny=\eta \text{ for some }A\subseteq \omega \text{ with }\|A\|_\varphi \ge c_\varphi(T,\eta)\right\}
$$
is comeager for all $\eta \in X$ such that $c_\varphi(T,\eta)>0$. 

In particular, there exists a point $x \in X$ for which $\eta$ is an $\mathsf{I}$-limit point of $\mathrm{orb}(x,T)$, where $\mathsf{I}:=\mathrm{Exh}(\varphi)$.
\end{thm}
\begin{proof}
Fix a vector $\eta \in X$ such that $c_\varphi(T,\eta)>0$. 
%Note that, for each $r>0$, the boundary $\partial B(\eta,r)$ is the sphere $\{x \in X: d(x,\eta)=r\}$, where $d$ stands for a round metric on $X$. 
%Now, pick a universal vector $y \in X$ and fix a decreasing sequence of positive reals $(r_k)$ with limit $0$ and such that $r_k\neq d(T^ny,\eta)$ for all $n,k \in \omega$ (notice that it is possible). 
%By the definition of $c_\varphi(T,\eta)>0$, 
By definition, 
there exists a sequence of universal vectors $(x_k)$ in $X$ such that 
\begin{equation}\label{eq:choiceconditions}
\|N(x_k,B(\eta,2^{-k}))\|_\varphi \ge c_\varphi(T,\eta)(1-2^{-k})
\end{equation}
for all $k \in \omega$. It follows by Proposition \ref{prop:lscsminclusion} that 
the set of vectors $x \in X$ such that $\|N(x,B(\eta,2^{-k}))\|_\varphi \ge \|N(x_k,B(\eta,2^{-k}))\|_\varphi$ for all $k \in \omega$ 
%$$
%\left\{x \in X: \|N(x,B(\eta,r_k))\|_\varphi \ge \|N(x_k,B(\eta,r_k))\|_\varphi \text{ for all }k \in \omega\right\}
%$$
is comeager. Together with \eqref{eq:choiceconditions}, it follows that the set 
$$
\left\{y \in X: \lim_{k\to \infty} \|N(y,B(\eta,2^{-k}))\|_\varphi \ge c_\varphi(T,\eta) \right\}
$$
is comeager. 
%Thanks to Theorem \ref{thm:univTemptycomeager}, the set of universal vectors $\mathrm{Univ}_T(\mathrm{Fin})$ is comeager. 
Hence, taking into account also Theorem \ref{thm:mareklimit}, there exists a comeager set of 
%universal 
vectors $y \in X$ with the property $\lim_{n \in A}T^n y=\eta$ for some set $A\subseteq \omega$ (which depends on $y$) such that $\|A\|_\varphi \ge  c_\varphi(T,\eta)$. This completes the proof. 
\end{proof}

The following easy lemma is folklore and we include its proof for the sake of completeness, cf. e.g. \cite[Claim 5.24]{MR4238631}:
\begin{lem}\label{lem:triviallemsynd}
Let $\mathsf{I}$ be an ideal on $\omega$. If $A \in \mathsf{I}^+$ and $B$ is syndetic then $A\cap (B-h) \in \mathsf{I}^+$ for some $h \in \omega$ with $h\le h_0$, where $h_0\ge 1$ stands for the maximum gap of $B$. 

If, in addition, $\mathsf{I}=\mathrm{Exh}(\varphi)$ for some lscsm $\varphi$, then it is possible to choose $h\le h_0$ such that $\|A\cap (B-h)\|_\varphi \ge \|A\|_\varphi/h_0$. 
\end{lem}
\begin{proof}
Since $\bigcup_{h\le h_0}(B-h)$ is cofinite then $\bigcup_{h\le h_0}(A \cap (B-h))$ is equal to $A$ modulo a finite set. Hence at least one element is $\mathsf{I}$-positive. The second part follows by the fact that $\|\cdot\|_\varphi$ is subadditive and invariant modulo finite sets.
\end{proof}

\begin{thm}\label{thm:sufficientforIuniversal}
Let $(X,T)$ be a universal dynamical system, where $X$ is a Banach space and $T$ is a homomorphism. 
Fix a lscsm $\varphi: \mathcal{P}(\omega) \to [0,\infty)$ such that $\|\cdot\|_\varphi$ is translation invariant and define $\mathsf{I}:=\mathrm{Exh}(\varphi)$. 
Also, suppose that there exists a dense set of uniformly recurrent vectors, and that $c_\varphi(T, 0)>0$.

Then $T$ is $\mathsf{I}$-universal.
\end{thm}
\begin{proof}
Thanks to Theorem \ref{cor:Ilimitswithcvarphi} there exists a comeager set $X_0$ of vectors $x \in X$ such that $\lim_{n \in A_x}T^nx=0$ for some $A_x\subseteq \omega$ with $\|A_x\|_\varphi \ge c_\varphi(T, 0)$. Note that, since $T$ is a continuous homomorphism, $\lim_{n \in A_x+m}T^nx=0$ holds too, for all $m \in \omega$. 
Since $T$ is universal and $X$ is a Banach space, there exists a countable base of nonempty open sets $(U_k)$. 
For each $k \in \omega$, it is possible to fix a uniformly recurrent vector $x_k \in X$ and a sufficiently small radius $r_k>0$ such that 
$$
B(x_k,2r_k) \subseteq U_k.
$$
Denote also $h_k$ the maximum gap of the syndetic set $B_k:=N(x_k, B(x_k,r_k))$. It follows by Lemma \ref{lem:triviallemsynd} that there exists an integer $i_{x,k} \in [0,h_k]$ such that 
$$
\left\|(A_x+i_{x,k}) \cap B_k\right\|_\varphi=
\left\|A_x \cap (B_k-i_{x,k})\right\|_\varphi \ge \frac{c_\varphi(T, 0)}{h_k}>0
$$
for all $x \in X_0$. At the same time, for all $x \in X_0$, all $k \in\omega$, and all sufficiently large integers $n \in (A_x+i_{x,k}) \cap B_k$, we have 
$$
T^n(x+x_k)=
T^{n}x+T^{n}x_k \in B(0,r_k)+B(x_k,r_k)\subseteq U_k.
$$
%$$
%T^n(x+T^{i_{x,k}}x_k)=T^nx+T^{n+i_{x,k}}x_k \in B(0,r_k)+B(x_k,r_k)\subseteq U_k.
%$$
Therefore $N(x+x_k, U_k)$ contains all sufficiently large integers in $(A_x+i_{x,k}) \cap B_k$, which is an $\mathsf{I}$-positive set. It follows by construction that 
$$
\bigcap\nolimits_{k \in\omega} (X_0+x_k) \subseteq \mathrm{Univ}_T(\mathsf{I}),
$$
so that there exists a comeager set of $\mathsf{I}$-universal vectors. 
\end{proof}

\begin{rmk}
It is clear that Theorem \ref{thm:sufficientforIuniversal} holds, with the same proof, for 
metrizable 
Baire topological groups $X$. 
% with a compatible round metric.
\end{rmk}

As an application, we obtain a generalization of \cite[Theorem 5.23]{MR4238631}: 
\begin{cor}\label{cor:extensionhilbertcharactrization}
Let $(X,T)$ be a linear dynamical system, where $X$ is a Banach space. 
Fix a lscsm $\varphi: \mathcal{P}(\omega) \to [0,\infty)$ such that $\|\cdot\|_\varphi$ is translation invariant and define $\mathsf{I}:=\mathrm{Exh}(\varphi)$. 
Also, suppose that there exists a dense set of uniformly recurrent vectors. 

Then the following are equivalent\textup{:}
\begin{enumerate}[label={\rm (\roman{*})}]
\item \label{item:1orbit22} $T$ is $\mathsf{I}$-hypercyclic\textup{;}
\item \label{item:2orbit22} $T$ is hypercyclic and $c_\varphi(T,0)>0$\textup{;}
\item \label{item:3orbit22} $T$ is hypercyclic and $\{x \in X: 0 \text{ is an }\mathsf{I}\text{-limit point of }\mathrm{orb}(x,T)\}$ is comeager\textup{;}
\item \label{item:4orbit22} There exists a hypercyclic vector $x$ such that $0$ is an $\mathsf{I}$-limit point of its orbit\textup{.}
\end{enumerate}
\end{cor}
\begin{proof}
\ref{item:1orbit22} $\implies$ \ref{item:2orbit22}. Suppose that $T$ is $\mathsf{I}$-hypercyclic. Reasoning as in \cite[Fact 4.4]{MR3255465}, we can rewrite equivalently
\begin{equation}\label{eq:definitioncT}
c_\varphi(T,0)=\sup \left\{ \| N(x, B(0,1))\|_\varphi: x \text{ is a hypercyclic vector}\right\}.
\end{equation}
(Indeed, by the homogeneity of $T$, the set of hypercyclic vectors is invariant to dilations.) However, the latter quantity is certainly positive if $T$ admits an $\mathsf{I}$-hypercyclic vector. 

\medskip 

\ref{item:2orbit22} $\implies$ \ref{item:3orbit22}. It follows by Theorem \ref{thm:mareklimit} and Theorem \ref{cor:Ilimitswithcvarphi}.

\medskip

\ref{item:3orbit22} $\implies$ \ref{item:4orbit22}. It follows by the fact that the set of hypercyclic vectors is comeager, see Theorem \ref{thm:univTemptycomeager}. 

\medskip 

\ref{item:4orbit22} $\implies$ \ref{item:2orbit22}. Fix a hypercyclic vector $x \in X$ and $A \in \mathsf{I}^+$ such that $\lim_{n \in A}T^nx=0$. Thanks to Theorem \ref{thm:mareklimit} and the definition of $c_\varphi(T,0)$, we obtain that 
$$
c_\varphi(T,0) \ge \lim_{k\to \infty} \|N(x,B(0,2^{-k})\|_\varphi \ge \|A\|_\varphi >0.
$$

\medskip

\ref{item:2orbit22} $\implies$ \ref{item:1orbit22}. It follows by Theorem \ref{thm:sufficientforIuniversal}.  
\end{proof}

It is worth noting that the equivalent definition of the parameter $c_\varphi(T,0)$ in \eqref{eq:definitioncT} does not rely on the full linearity of $T$: indeed, it is sufficient that $T(\alpha y)=\alpha T(y)$ for all vectors $y \in X$ and all scalars $\alpha$. A similar comment goes to Corollary \ref{cor:flkjshlkjghkj} below. 

As another application, 
%of Proposition \ref{prop:lscsminclusion}, 
we obtain a result related to \cite[Proposition 14]{MR3334899}. 
%To this aim, following Nathanson \cite{MR0407812}, we say that a metric space $X$ is \emph{round} if the closure of every open ball is the corresponding closed ball; hence, examples of Baire round metric spaces include all Banach spaces. 
\begin{cor}\label{cor:kfjdshlgkjhsdlkhj}
Let $(X,T)$ be a universal dynamical system, where $X$ is a Baire 
%round 
metric space and the metric $d$ is unbounded. Also, fix $\eta \in X$ and suppose that 
$$
\kappa:=\inf_{r >0} \sup \left\{\|\omega \setminus N(x,B(\eta,r))\|_\varphi: x \text{ is a universal point}\right\}>0.
$$
and fix a lscsm $\varphi: \mathcal{P}(\omega) \to [0,\infty)$ such that $\|\cdot\|_\varphi$ is translation invariant.

Then there exists a comeager set of vectors $y \in X$ such that 
%$$
%\lim_{n\in A} d(T^n y,\eta)=+\infty \,\,\,\text{ for some }A\subseteq \omega \text{ with }\|A\|_\varphi\ge \kappa.
%$$
$\lim_{n\in A} d(T^n y,\eta)=+\infty$  for some $A\subseteq \omega $ with $\|A\|_\varphi\ge \kappa$. 
%$$
%\{y \in X: \lim_{n\in A} \|T^n y\|=+\infty \text{ for some }A\subseteq \omega \text{ with }\|A\|_\varphi\ge \kappa \}
%$$
%is comeager. 
\end{cor}
\begin{proof}
For each real $r >0$, define the open set $U_r:=X\setminus \overline{B(\eta,r)}$. By the standing hypothesis, for each $r>0$ and $i \in \omega$ there exists a universal point $x_{r,i}$ such that $\|N(x_{r,i},U_r)\|_\varphi \ge \kappa(1-2^{-i})$. 
%Replacing each $r$ with another real in $[r,r+1]$, if necessary, we can suppose without loss of generality that $d(T^n x_{r,i},\eta) \neq r$ for all $n \in \omega$. 
Hence, we get by Proposition \ref{prop:lscsminclusion} that 
$$
\bigcap\nolimits_{i,h \in \omega,\, h>0}\left\{y \in X: \|N(y,U_h)\|_\varphi \ge \|N(x_{h,i},U_h)\|_\varphi \right\}
$$
is comeager. Therefore there exists a comeager set of vectors $y \in X$ such that 
$$
\|\{n \in \omega: d(T^ny,\eta)\ge r\}\|_\varphi \ge \kappa
$$
for all $r>0$. The conclusion follows applying Theorem \ref{thm:mareklimit} to the space $\mathbb{R}\cup \{\infty\}$ and each real sequence $(d(T^ny,\eta): n \in \omega)$. 
\end{proof}

In the setting of linear dynamical systems, the above corollary simplifies: 
\begin{cor}\label{cor:flkjshlkjghkj}
Let $(X,T)$ be a hypercyclic linear dynamical system, where $X$ is a nonzero Banach space. Fix a hypercyclic vector $x \in X$ and a lscsm $\varphi: \mathcal{P}(\omega) \to [0,\infty)$ such that $\|\cdot\|_\varphi$ is translation invariant. 

Then there exists a comeager set of vectors $y \in X$ such that $\lim_{n\in A} \|T^ny\|=+\infty$  for some $A\subseteq \omega $ with $\|A\|_\varphi\ge \|\,\omega\setminus N(x,B(0,1))\|_\varphi$. 
\end{cor}
\begin{proof}
Define $\gamma:=\|\,\omega\setminus N(x,B(0,1))\|_\varphi$. If $\gamma=0$, then it is enough to choose every hypercyclic vector $y \in X$ (hence, in a comeager set). Otherwise, suppose that $\gamma>0$. Since $N(x,B(0,r))=N(x/r, B(0,1))$ and $x/r$ is still hypercyclic for each $r>0$, it is enough to apply Corollary \ref{cor:kfjdshlgkjhsdlkhj} with $\eta=0$ and $\kappa=\gamma>0$. 
\end{proof}

%%%%%%%%%%%%%%%%%%%%%%%%%%%

%\clearpage
%%%%%%%%%%%%%%%%%%%%%%%%%%

%\section{Ansari's theorem for arithmetic ideals}

\subsection{Ansari-type result} 
An ideal $\mathsf{I}$ on $\omega$ is said to be \emph{arithmetic} if 
$$
S \in \mathsf{I}
 \quad \text{ if and only if }\quad 
k\cdot S+h \in \mathsf{I}
$$
for all $S\subseteq \omega$ and $k,h \in \omega$ with $k\neq 0$. 
The class of arithmetic ideals is very large. For, let $\mu^\star: \mathcal{P}(\omega) \to \mathbb{R}$ be a monotone subadditive map such that $\mu^\star(\omega)=1$ and $\mu^\star(k\cdot S+h)=\mu^\star(S)$ for all $S\subseteq \omega$ and $k,h \in \omega$ with $k\neq 0$ (which can be regarded as an abstract formulation of \textquotedblleft upper density,\textquotedblright\, see \cite{MR4054777}). Then the family
$$
\mathsf{I}(\mu^\star):=\left\{S\subseteq \omega: \mu^\star(S)=0\right\}
$$
is an arithmetic ideal. Examples of these maps $\mu^\star$ include the upper asymptotic, upper Banach, upper
analytic, upper logarithmic, upper P\'{o}lya, and upper Buck densities, see \cite[Section 6 and Examples 4, 5,
6, and 8]{MR4054777}. In particular, $\mathsf{L}$, $\mathsf{Z}$, and $\mathsf{B}$ are arithmetic ideals. It is worth mentioning that also rather exotic examples of such maps $\mu^\star$ exist (e.g., finitely additive ones, see \cite[Remark 3]{MR4054777}).

%\textcolor{red}{NUOVI ESEMPI E IL LAVORO DI GROSSE ERDMANN JFA}
\begin{thm}\label{thm:generalizedansari}
Let $(X,T)$ be a linear dynamical system, where $X$ is a Fr\'{e}chet space. Let also $\mathsf{I}$ be an arithmetic ideal on $\omega$ and fix a nonzero $k \in \omega$. 
Then $$\mathrm{Univ}_{T^k}(\mathsf{I})=\mathrm{Univ}_{T}(\mathsf{I}).$$% for all integers $k\ge 1$. 
\end{thm}
\begin{proof} 
%The case $\mathsf{I}=\mathrm{Fin}$ has been proved by Ansari in \cite[Theorem 1]{MR1319961}. %Here we proceed on the same lines of \cite[Theorem 4.7]{MR2231886}. 
We start by proving the easy inclusion $\mathrm{Univ}_{T^k}(\mathsf{I}) \subseteq \mathrm{Univ}_{T}(\mathsf{I})$. This is trivial if $\mathrm{Univ}_{T^k}(\mathsf{I})=\emptyset$. Otherwise, fix a vector $x \in X$ which is $\mathsf{I}$-hypercyclic for $T^k$, let $U\subseteq X$ be a nonempty open set. Since 
$$
N_T(x,U)=
\bigcup_{i=0}^{k-1}\left(k\cdot N_{T^k}(T^ix,U)+i\right).
$$
we obtain that $N_T(x,U)$ contains (some translation of) $k\cdot N_{T^k}(x,U) \in \mathsf{I}^+$. By the arbitrariness of $U$, $x$ is $\mathsf{I}$-hypercyclic for $T$. 

Conversely, we need to show that $\mathrm{Univ}_{T}(\mathsf{I}) \subseteq \mathrm{Univ}_{T^k}(\mathsf{I})$. 
For, suppose that the left hand side is nonempty and let $x$ be an $\mathsf{I}$-hypercyclic vector for $T$ (hence, in particular, $x$ is hypercyclic). 
It follows by Lemma \ref{lem:easylem} and Ansari's theorem \cite[Theorem 1]{MR1319961} that 
\begin{equation}\label{eq:inclusionansari}
\{x,Tx,\ldots,T^{k-1}x\}\subseteq \mathrm{Univ}_{T}(\mathrm{Fin})=\mathrm{Univ}_{T^k}(\mathrm{Fin}).
\end{equation}
Let $U\subseteq X$ be a nonempty open set. Now, fix $y \in U$ and a neighborhood $V$ of $0$ such that $y+V+V\subseteq U$, 
%see e.g. \cite[p. 10]{MR1157815}. 
see e.g. \cite[Chapter 3]{MR0000265}.
It follows by \eqref{eq:inclusionansari} that we can pick %an integer 
$$
a_i \in N_{T^k}(T^ix,y+V)
$$
for each $i \in \{0,\ldots,k-1\}$, thus $T^{n_i}x\in y+V$ with $n_i:=a_ik+i$. By the continuity of $T$, there exists a neighborhood $W$ of $0$ such that $T^{n_i}[W] \subseteq V$ for all $i \in \{0,\ldots,k-1\}$. 
At this point, we have $N_{T}(x,x+W) \in \mathsf{I}^+$ since $x$ is an $\mathsf{I}$-hypercyclic vector of $T$. Hence there exists $h \in \{0,\ldots,k-1\}$ such that 
$$
S:=N_{T}(x,x+W) \cap (k\cdot \omega+h) \in\mathsf{I}^+.
$$ 
More explicitly, $S$ can be rewritten as $\{n \in \omega: T^nx-x \in W \text{ and }n\equiv h\bmod{k}\}$. Set $j:=k-h$ if $h\neq 0$ and $j:=0$ if $h=0$, so that $j+h \in \{0,k\}$ and $j \in \{0,\ldots,k-1\}$. To conclude the proof, pick an integer $n \in S$. Then 
\begin{displaymath}
T^{n_j}T^nx=T^{n_j}(T^nx-x)+T^{n_j}x \in T^{n_j}[W]+(y+V) \subseteq y+V+V\subseteq U.
\end{displaymath}
At the same time, setting $s:=a_j+\frac{1}{k}(n+j)$, we have 
$$
T^{n_j}T^nx=T^{a_jk+j+n}x=(T^k)^{s}x.
$$
Therefore $N_{T^k}(x,U)$ contains $a_j+\nicefrac{1}{k}\cdot (S+j)$, which is an $\mathsf{I}$-positive set since $\mathsf{I}$ is arithmetic. Therefore $x$ is an $\mathsf{I}$-hypercyclic vector for $T^k$.  
\end{proof}

\begin{rmk}
More generally, 
%An inspection of the proof reveals that 
the analogue of Theorem \ref{thm:generalizedansari} holds replacing $\mathsf{I}^+$ with an arbitrary Furstenberg family $\mathsf{F}$ satisying the following property: 
\begin{equation}\label{eq:hypercyclicitycondition}
\forall S \in \mathsf{F},\forall a_0,\ldots,a_{k-1} \in \omega, \quad 
\bigcup_{h=0}^{k-1}(a_h+\nicefrac{1}{k}\cdot (S_h-h)) \in \mathsf{F},
\end{equation}
where $S_h:=S\cap (k\cdot \omega +h)$ for all $h \in \{0,1,\ldots,k-1\}$. 
%$$
%\forall S \in \mathsf{F}, \exists h \in \omega, \quad 
%S \cap (k\cdot \omega+h) \in \mathsf{F}.
%$$ 
%for every $S \in \mathsf{F}$ there exists $h \in \omega$ such that $S \cap (k\cdot \omega+h) \in \mathsf{F}$. 
(Recall that the integer $k$ is fixed.) This includes, for instance, also the case of $\mathsf{F}$ being the family of positive lower asymptotic density sets, as in \cite[Theorem 4.7]{MR2231886}. 

It is worth mentioning that Bonilla, Grosse-Erdmann, L\'{o}pez-Mart\'{\i}nez, and Peris very recently proved in \cite[Theorem 8.8]{MR4489276} that, if $T$ is linear, the analogue of Theorem \ref{thm:generalizedansari} holds for Furstenberg families $\mathsf{F}$ satisfying the following requirements: 
%(i) $S+h \in \mathsf{F}$ for every $S \in \mathsf{F}$ and $h \in \omega$, 
(i) $S \in \mathsf{F}$ if and only if $k\cdot S \in \mathsf{F}$ for every $S\subseteq \omega$, and 
(ii) For every finite cover $\{I_1,\ldots,I_q\}$ of $\omega$ and for every $h_1,\ldots,h_q \in \omega$ and for all $S \in \mathsf{F}$, then $\bigcup_j ((S\cap I_j)+h_j) \in \mathsf{F}$. 
However, it is not difficult to see that latter conditions are strictly stronger than ours. %For instance, set $E:=2\cdot \omega$. Then, for each nonzero $k \in \omega$, the Furstenberg family $\mathsf{F}=\{S\subseteq \omega: S \cap E \in \mathrm{Fin}^+\}$ satisfies our condition, but it fails the latter one because $E \in \mathsf{F}$ and $E+1 \notin \mathsf{F}$. 
For instance, let $\mathsf{F}$ be the Furstenberg family of sets with positive lower Buck density, see e.g. \cite[Section 4]{MR4054777}, which can be written explicitly as 
$$
\mathsf{F}:=\{S\subseteq \omega: k\cdot \omega+h\subseteq S \text{ for some  }k,h \in \omega \text{ with }k\neq 0\}. 
$$
Then $\mathsf{F}$ satisfies property \eqref{eq:hypercyclicitycondition}, but it fails the latter condition (ii): indeed, set  $I_1:=\{n!+n: n\ge 1\}$ and $I_2:=\omega\setminus I_1$ and note that both $I_1$ and $I_2$ do not contain any infinite arithmetic progression $k\cdot \omega+h$, therefore $\omega\in \mathsf{F}$ while $(I_1+1) \cup I_2=I_2 \notin \mathsf{F}$. 

% because $S=\omega\in \mathsf{F}$ while $(I_1+1) \cup I_2=I_2 \notin \mathsf{F}$, where $I_1:=\{n!+n: n\ge 1\}$ and $I_2:=\omega\setminus I_1$; indeed, it is not difficult to check that both $I_1$ and $I_2$ do not contain any infinite arithmetic progression $k\cdot \omega+h$. 
%, cf. also \cite[]{}.  

In the same work, the authors studied the analogous result from a recurrence point of view, see  \cite[Theorem 8.7]{MR4489276}.
\end{rmk}

\begin{rmk}
If we do not consider the base case $\mathsf{I}=\mathrm{Fin}$, the proof of Theorem \ref{thm:generalizedansari} does not use the full linearity of $T$. Indeed, taking into account \cite[Exercise 6.1.7]{MR2919812}, the analogue of Theorem \ref{thm:generalizedansari} holds in the nonlinear setting for dynamical systems $(X,T)$ where $X$ is a metric group and $T$ is a homomorphism, under the additional hypothesis that $\mathrm{Univ}_T(\mathrm{Fin})$ contains a connected dense subset. 
\end{rmk}

\bibliographystyle{amsplain}
\bibliography{ideale}
\end{document}